\documentclass[11pt]{amsart}
\newcommand{\sfrac}[2]{{\textstyle\frac{#1}{#2}}}
\usepackage{amsmath,amssymb}
\usepackage{graphicx}  
\usepackage{color}

\numberwithin{equation}{section}

\newcommand{\Ex}{\mathbb{E}}
\renewcommand\Ex{\operatorname{\mathbb E}{}} 
 \renewcommand{\Pr}{{\mathbb{P}}}

 \newcommand{\eps}{\varepsilon}

 \newcommand\CTCS{\operatorname{CTCS}}
\newcommand\DTCS{\mathrm{DTCS}}
\newtheorem{Lemma}{Lemma}[section]
\newtheorem{lemma}[Lemma]{Lemma}
\newtheorem{Theorem}[Lemma]{Theorem}
\newtheorem{theorem}[Lemma]{Theorem}
\newtheorem{Proposition}[Lemma]{Proposition}

\theoremstyle{definition}
\newtheorem{Remark}[Lemma]{Remark}
\newtheorem{remark}[Lemma]{Remark}
\newtheorem{example}[Lemma]{Example}

\newenvironment{romenumerate}[1][-10pt]{
\addtolength{\leftmargini}{#1}\begin{enumerate}
 }{\end{enumerate}}

 \newcommand{\var}{\mathrm{var}}

\newcommand\nnn{^{[n]}}
\newcommand\bbN{\mathbb N}
\newcommand\bbR{\mathbb R}
\newcommand\bbC{\mathbb C}

\newcommand\setn{\set{1,\dots,n}}
\newcommand\set[1]{\ensuremath{\{#1\}}}

\newcommand\xpar[1]{(#1)}
\newcommand\bigpar[1]{\bigl(#1\bigr)}
\newcommand\Bigpar[1]{\Bigl(#1\Bigr)}

\newcommand\bigsqpar[1]{\bigl[#1\bigr]}

\newcommand\Bigsqpar[1]{\Bigl[#1\Bigr]}

\newcommand\lrsqpar[1]{\left[#1\right]}
\newcommand\cpar[1]{\{#1\}}
\newcommand\bigcpar[1]{\bigl\{#1\bigr\}}

\newcommand\bigabs[1]{\bigl\lvert#1\bigr\rvert}
\newcommand\Bigabs[1]{\Bigl\lvert#1\Bigr\rvert}

\newcommand\lrabs[1]{\left\lvert#1\right\rvert}
\newcommand\xfrac[2]{#1/#2}

\newcommand\KX{K}

\newcommand\XP{P}

\newcommand\ntoo{\ensuremath{{n\to\infty}}}

\newcommand\stoo{\ensuremath{{s\to\infty}}}

\newcommand\oi{\ensuremath{[0,1]}}

\newcommand\oio{\ensuremath{(0,1)}}

\newcommand\qooo{(0,\infty)}

\newcommand\cXt{\XP_{t,1}}

\newcommand\gd{\delta}

\newcommand\gam{\gamma}

\newcommand\gG{\Gamma}

\newcommand\gl{\lambda}
\newcommand\gL{\Lambda}
\newcommand\go{\omega}

\newcommand\gs{\sigma}

\newcommand\gss{\sigma^2}
\newcommand\gt{\tau}

\newcommand\gu{\upsilon}
\newcommand\gU{\Upsilon}

\newcommand\cD{\mathcal D}

\newcommand\cL{{\mathcal L}}

\newcommand{\sumi}{\sum_{i=1}^\infty}
\newcommand{\sumj}{\sum_{j=1}^\infty}
\newcommand{\sumko}{\sum_{k=0}^\infty}
\newcommand{\sumk}{\sum_{k=1}^\infty}

\newcommand\intoo{\int_0^\infty}
\newcommand\intoooo{\int_{-\infty}^\infty}
\newcommand\dd{\,\mathrm{d}}
\newcommand\ddx{\mathrm{d}}

\newcommand\intoi{\int_0^1}
\newcommand\intox{\int_0^x}
\newcommand\qw{^{-1}}
\newcommand\Res{\operatorname{Res}}

\newcommand\nux{\nu_\Delta}
\newcommand\tnux{\widetilde{\nux}}

\newcommand\qww{^{-2}}
\newcommand\ii{\mathrm{i}}

\newcommand\downto{\searrow}
\newcommand\upto{\nearrow}

\newcommand{\refT}[1]{Theorem~\ref{#1}}
\newcommand{\refTs}[1]{Theorems~\ref{#1}}

\newcommand{\refL}[1]{Lemma~\ref{#1}}
\newcommand{\refLs}[1]{Lemmas~\ref{#1}}
\newcommand{\refR}[1]{Remark~\ref{#1}}

\newcommand{\refS}[1]{Section~\ref{#1}}
\newcommand{\refSs}[1]{Sections~\ref{#1}}
\newcommand{\refSS}[1]{Section~\ref{#1}}

\newcommand{\refApp}[1]{Appendix~\ref{#1}}

\newcommand\lhs{left-hand side}
\newcommand\rhs{right-hand side}

\newcommand\intgs{\int_{\gs-\ii\infty}^{\gs+\ii\infty}}
\newcommand\M[1]{\widetilde{#1}}
\newcommand\Mf{\M f}

\newcommand\Mr{\M r}
\newcommand\Mrprime{\M{r}\,{}'}
\newcommand\Mmu{\M\mu}
\newcommand\MgU{\M\gU}

\newcommand\Mfn{\M{f}_n}
\newcommand\MHn{\M{H}_n}

\newcommand\marginal[1]{\marginpar[\raggedleft\tiny #1]{\raggedright\tiny#1}}

\newcommand\REM[1]{{\raggedright\texttt{[#1]}\par\marginal{XXX}}}

\newcommand{\tend}{\longrightarrow}
\newcommand\dto{\overset{\mathrm{d}}{\tend}}
\newcommand\pto{\overset{\mathrm{p}}{\tend}}

\newcommand\E{\Ex}

\newcommand\cDo{\cD^\circ}

\newcommand\gsx{\gs_*}
\newcommand\gLX{\gL^*}
\newcommand\rhox{\rho_*}

\newcommand\hnu{\widehat\nu}

\newcommand\hF{\widehat{F}}
\newcommand\hf{\widehat{f}}

\newcommand\hks{\widehat {k_s}}

\begingroup
  \divide\count255 by 60
  \multiply\count255 by -60
  \advance\count255 by \time
  \ifnum \count255 < 10 \xdef\klockan{\the\count1.0\the\count255}
  \else\xdef\klockan{\the\count1.\the\count255}\fi
\endgroup

\begin{document}

\title[Mellin analysis of a random tree]
{The Critical Beta-splitting Random Tree IV: Mellin analysis of Leaf Height}

 \author{David J. Aldous}
\address{Department of Statistics,
 367 Evans Hall \#\  3860,
 U.C. Berkeley CA 94720}
\email{aldousdj@berkeley.edu}
\urladdr{www.stat.berkeley.edu/users/aldous.}

\author{Svante Janson}
\thanks{SJ supported by the Knut and Alice Wallenberg Foundation
and
the Swedish Research Council
}
\address{Department of Mathematics, Uppsala University, PO Box 480,
SE-751~06 Uppsala, Sweden}
\email{svante.janson@math.uu.se}
\newcommand\urladdrx[1]{{\urladdr{\def~{{\tiny$\sim$}}#1}}}
\urladdrx{www2.math.uu.se/~svante}

\date{December 14, 2024} 

 \begin{abstract}
 In the critical beta-splitting model of a random $n$-leaf rooted tree, clades are recursively split into sub-clades, and a clade of $m$ leaves is split into sub-clades 
 containing  $i$ and $m-i$ leaves  with probabilities $\propto 1/(i(m-i))$. 
 The height of a uniform random leaf can be represented as the absorption
 time of a certain {\em harmonic descent} Markov chain.  
 Recent work on these heights $D_n$ and $L_n$ (corresponding to discrete or continuous versions of the tree)
 has led to quite sharp expressions for their asymptotic distributions, based on their Markov chain description.
 This article gives even sharper expressions, based on
 an $n \to \infty$ limit tree structure described via exchangeable random partitions
 in the style of Haas et al (2008).
 Within this structure, calculations of moments lead to expressions for Mellin transforms,
 and then via Mellin inversion
 we obtain sharp estimates for the expectation, variance, Normal approximation and large deviation
 behavior of $D_n$.
 \end{abstract}
 
\maketitle

  
{\tt {\bf Keywords:}
Exchangeable partition, Markov chain, Mellin transform, random tree, subordinator.}

MSC 60C05; 05C05, 44A15, 60G09.

\tableofcontents


\section{Introduction}\label{S:intro}

\subsection{The tree model}
  \label{sec:tree}
A more detailed account of the model, with graphics, is given in \cite{beta3-arxiv}, and this article is in some ways a continuation of that article.

For $m \ge 2$, consider the probability distribution $(q(m,i),\ 1 \le i \le m-1)$
constructed to be proportional to $\frac{1}{i(m-i)}$.
 Explicitly (by writing $\tfrac{1}{i(m-i)}=\bigl(\tfrac{1}{i}+\tfrac{1}{m-i}\bigr)/m$)
\begin{equation}\label{01}
q(m,i)=\tfrac{m}{2h_{m-1}}\cdot\tfrac{1}{i(m-i)},\,\,1\le i\le m-1,
\end{equation}
where $h_{m-1}$ is the harmonic sum $\sum_{i=1}^{m-1}1/i$. 

Now fix $n \ge 2$.
Consider the process of constructing a random tree by recursively splitting the integer interval 
$[n] = \{1,2,\ldots,n\}$ of ``leaves" as follows.
First specify that there is a left edge and a right edge at the root,
 leading to a left subtree which will have the\footnote{$G$ for {\em gauche} (left) because we use $L_n$ for leaf height.}  $G_n$ leaves $\{1,\ldots,G_n\}$
 and a right subtree which will have the remaining $R_n = n - G_n$ leaves $\{G_n + 1,\ldots, n\}$, where $G_n$ 
 (and also $R_n$, by symmetry) has distribution $q(n,\cdot)$. 
 Recursively, a subinterval with $m\ge 2$ leaves is split into two subintervals of random size
  from the distribution $q(m,\cdot)$. 
  Continue until reaching intervals of size $1$, which are the leaves.
This process has a natural tree structure. 
In this discrete-time construction, which we call $\DTCS(n)$, we regard the edges of the tree as having length $1$.
It turns out (see e.g. \cite{beta2-arxiv})
to be natural to consider also the continuous-time construction $\CTCS(n)$ in which 
a size-$m$ interval is split at rate $h_{m-1}$, that is after an Exponential$(h_{m-1})$ holding time.
Once such a tree is constructed, it is natural to identify ``time" with ``distance": a leaf that appears at time $t$ has
{\em height} $t$.
Of course the discrete-time model is implicit within the continuous-time model, and a leaf in $\DTCS(n)$ which appears after 
$\ell$ splits has {\em hop-height} $\ell$.

Finally, our results do not use the leaf-labels $\{1,2,\ldots, n \}$ in the interval-splitting construction.
Instead they involve a uniform {\em random} leaf. 
Equivalently, one could take a uniform random permutation of labels and then talk about
the leaf with some arbitrary label.


Many aspects of this model can be studied 
by different techniques, see \cite{beta2-arxiv} for an overview.
In this article we focus on one aspect and one methodology, as follows.

\subsection{Leaf heights}
\label{sec:Lh}
It is an elementary calculation \cite{beta1} to show that the discrete time process described by
\begin{quote}
In the path from the root to a uniform random leaf in $\DTCS(n)$, consider  at each step the size (number of leaves) of the sub-tree rooted at the current position
\end{quote}
is the discrete time Markov chain $(X^{disc}_t, t = 0, 1, 2, \ldots)$ on states $\{1,2,3,\ldots\}$   with transition probabilities
\begin{equation}
q^*_{m,i} := \sfrac{1}{h_{m-1}} \ \sfrac{1}{m-i}, \ 1 \le i \le m-1, m \ge 2 .
\end{equation}
Similarly,  the continuous time process 
described by
\begin{quote}
Move at speed one along the edges of the path from the root to a uniform random leaf of $\CTCS(n)$, and consider  at each time the size (number of leaves) of the sub-tree rooted at the current position
\end{quote}
is the continuous time Markov chain $(X^{cont}_t, t \ge 0)$ 
with transition rates
\begin{equation}
\lambda_{m,i} := \sfrac{1}{m-i}, \ 1 \le i \le m-1, m \ge 2 .
\end{equation}
Each process is absorbed at state $1$.
So if we define

$D_n :=$ hop-height of a uniform random leaf of $\DTCS(n)$

$L_n :=$ height of a uniform random leaf of $\CTCS(n)$

\noindent
then these are the same as the appropriate Markov chain absorption time
\begin{eqnarray}
D_n & = & \inf\{t: X^{cont}_t = 1 \mid X^{cont}_0 = n\} \label{Dndef}\\
L_n & = & \min\{t: X^{disc}_t = 1 \mid X^{disc}_0 = n\} \label{Lndef}.
\end{eqnarray}

We call these Markov chains the {\em harmonic descent (HD)} chains.  
Recent studies of $D_n$ and $L_n$ \cite{HDchain,beta1,iksanovCLT,iksanovHD,kolesnik}
have been based on the Markov chain representation (\ref{Dndef}, \ref{Lndef}).
Note that to study $n \to \infty$ asymptotics for these chains,  one cannot directly formalize the idea of starting a hypothetical version of the chain from
$+\infty$ at time  $- \infty$. 
However in the underlying random tree model $\CTCS(n)$ there is a more sophisticated 
way to formalize a limit tree structure $\CTCS(\infty)$ via exchangeable random partitions (Section \ref{sec:exch}).
In the context of our model, 
this methodology was first used in \cite{beta3-arxiv}, and  this continuation article will show that one can then apply Mellin transform techniques 
to obtain asymptotic estimates that are sharper than those obtained by previous methods.
Here are summaries of our main results.

\subsection{Summary of results}
\label{sec:summary}
Let $\psi$ be the digamma function (see Section \ref{SSpsi}). 
Let $0>s_1>s_2>\dots$ be the negative roots of $\psi(s)=\psi(1)$.
Recall that $\zeta(2) = \pi^2/6$ and $\zeta(3) \doteq 1.202$,
and note that $\sim$ in the results below denotes asymptotic expansion
(see \refS{SSasy}).

\begin{Theorem}\label{TD1}
As $n \to \infty$
  \begin{align}\label{aw21}
  \Ex[D_n] \sim 
\frac{6}{\pi^2}\log n
+\sum_{i=0}^\infty c_{i} n^{-i}
+\sum_{j=1}^\infty\sum_{k=1}^\infty c_{j,k}\,n^{-|s_j|-k}
\end{align}
for some coefficients $c_i$ and $c_{j,k}$ that can be found explicitly;
in particular, $c_1=-3/\pi^2$ and 
\begin{equation}
c_0 =  {\frac{\zeta(3)}{\zeta(2)^2}+\frac{\gam}{\zeta(2)}} \doteq  0.795155660439.
\label{c01}
\end{equation}
\end{Theorem}
Theorem \ref{TD1} is proved in Sections \ref{sec:EDfull} and \ref{SED2}.
It improves on \cite[Theorem 1.2 and Proposition 2.3]{beta1} which gave the initial terms 
$\frac{6}{\pi^2}\log n + c_0 + c_1n^{-1}$ 
with the explicit formula for $c_1$ but not the formula\footnote{Before knowing the exact value of $c_0$, numerics gave an estimate that agrees with \eqref{c01} to 10 places.} 
for $c_0$.
The discussion of the  \emph{h-ansatz} in \cite{beta1} assumes that only integer  powers of $1/n$ should appear in the expansion \eqref{aw21}, 
but in fact (surprisingly?)\ the spectrum of powers of $n$ appearing is $\set{-i:i\ge0}\cup\set{-(|s_j|+k):j\ge1,k\ge1}$.

Note that there is a simple recurrence for $ \Ex[D_n]$.
The theme of \cite{beta1} was to exploit ``the recurrence method", that is to take a sequence defined by a recurrence
and then upper and lower bound the unknown sequence by known sequences.
This method was used in \cite{beta1} for many of the problems in this paper, as indicated in the references below.

\begin{Theorem}\label{TL1}
As $n \to \infty$
\begin{multline}\label{sw2x1}
  \E[L_n]\sim
\frac{3}{\pi^2}\log^2n 
+ \Bigpar{\frac{\zeta(3)}{\zeta(2)^2}+\frac{\gam}{\zeta(2)}}\log n
+b_0
\\
+\sumk a_k n^{-k}\log n
+\sumk b_k n^{-k}
+ \sumj \sumk c_{j,k}n^{-|s_j|-k}
\end{multline}
for some computable constants $a_k$, $b_k$, $c_{j,k}$;
in particular,
\begin{align}\label{hw211}
  b_0 = 
  \frac{3\gam^2}{\pi^2}
+\frac{\zeta(3)}{\zeta(2)^2}\gam+
\frac{\zeta(3)^2}{\zeta(2)^3}
+\frac{1}{10}
\doteq
0.78234
.\end{align}
\end{Theorem}
Theorem \ref{TL1} is proved in Sections \ref{SEL1} and \ref{SEL2}.
The first term $\frac{3}{\pi^2}\log^2n$ was observed long ago in \cite{me_clad}.
Using the recurrence method, the coefficient for $\log n$ 
was found in  \cite[Theorem 1.2]{beta1}; that coefficient 
equals the constant term $c_0$ in the asymptotic expansion \eqref{aw21} of
$\E[D_n]$.

\begin{theorem}\label{TvarD1}
As $n \to \infty$
  \begin{align}\label{kb91}
\var[D_n]
=
\frac{2\zeta(3)}{\zeta(2)^3}\log n
+ \frac{2\zeta(3)}{\zeta(2)^3}\gamma
+\frac{5\zeta(3)^2}{\zeta(2)^4}
-\frac{18}{5\pi^2}
+ O\Bigpar{\frac{\log n}{n}}
.\end{align}
\end{theorem}
Theorem \ref{TvarD1} is proved in Section \ref{SSEDk}.
The leading term $\frac{2\zeta(3)}{\zeta(2)^3}\log n$ was found by in \cite[Theorem 1.1]{beta1} by the recursion method.
Higher moments of $D_n$ are discussed in Section \ref{SSEDk}.

\begin{theorem}\label{Tmgf1}
For $-\infty < z < 1$ there is a 
unique real number $\rho(z)$ in $(-1,\infty)$ satisfying
 $ \psi\bigpar{1+\rho(z)}-\psi(1)=z$.
Then
\begin{align}\label{qz31}
\E[e^{zD_n}] &
=
\frac{-z\gG(-\rho(z))}{\psi'(1+\rho(z))}\frac{\gG(n)}{\gG(n-\rho(z))}
+O\bigpar{n^{-\gsx}}
\end{align}
and
\begin{align}\label{qz331}
\E[e^{zD_n}] &
=
\frac{-z\gG(-\rho(z))}{\psi'(1+\rho(z))}n^{\rho(z)}
\cdot\bigpar{1+O\bigpar{n^{-\min(1,\gsx+\rho(z))}}}
\end{align}
where 
$\gsx   \doteq 1.457$
is defined at \eqref{gs1}.
Furthermore, \eqref{qz31} holds uniformly for $z<1-\gd$ for any $\gd>0$,
and \eqref{qz331} holds uniformly for $z$ in a compact subset of $(-\infty,1)$.
\end{theorem}
Theorem \ref{Tmgf1} is proved in Section \ref{SSmgfreal}.
It improves on bounds in  \cite[Section 2.7]{beta1}.
As a corollary of \refT{Tmgf1}, we obtain a new proof of the following CLT.
This CLT has been proved in \cite[Theorem 1.7]{beta1} via the recursion method (applied to moment generating functions) 
and in \cite{beta2-arxiv} via the martingale CLT.
We prove Theorem \ref{TNorm1} in Section \ref{sec:CLT}.
\begin{Theorem}
\label{TNorm1}
\begin{align}\label{tnorm11}
  \frac{D_n - \mu \log n}{\sqrt{\log n}} \dto  
\mathrm{Normal}(0, \sigma^2) \ \mbox{ as } n \to \infty 
\end{align}
where
\begin{align}\label{tnorm21}
 \mu := 1/\zeta(2) = 6/\pi^2 \doteq 0.6079 ; \quad 
\sigma^2 := 2 \zeta(3)/\zeta(2)^3 \doteq 0.5401
.\end{align}
\end{Theorem}
Although not discussed in this article,
there is a parallel CLT for $L_n$, which has been proved in several quite different ways.
In \cite[Theorem 1.7]{beta1} using the recursion method. 
In \cite{kolesnik} using the general {\em contraction method}  \cite{NR04}.
In \cite{iksanovCLT} using known results \cite{GI12} in the theory of {\em regenerative composition structures}.


Another corollary of \refT{Tmgf1} is the following large deviation result, proved in Section \ref{sec:LD}.
\begin{theorem}\label{Tdev1}
As \ntoo, we have:
  \begin{align}\label{tdev11}
  \Pr(D_n<x\log n)& = n^{-\gLX(x)+o(1)}, \qquad\text{if}\quad 0<x\le x_0,
\\ \label{tdev21}
\Pr(D_n>x\log n)& =n^{-\gLX(x)+o(1)}, \qquad\text{if}\quad x_0\le x< x_1,
\\   \label{tdev31} 
  \Pr(D_n>x\log n)& \le n^{-\gLX(x)+o(1)}, \qquad\text{if}\quad x\ge x_1,
  \end{align}
  where $x_0 = 1/\zeta(2), \ x_1 = 1/(\zeta(2)-1)$ and $\Lambda^*$ is defined at \eqref{qd6}.
\end{theorem}

\refT{Tdev1} improves estimates for the upper tail in \cite[Theorem 1.4]{beta1} and \cite{beta2-arxiv}.

\begin{remark}
Note that for $L_n$, we give a result only for the mean.
We could not prove other results, for example for the variance,
by the methods of this article (see \cite{beta1} for results by the
recurrence method). 
The difficulty is that for higher moments of $L_n$
we have been unable to find a representation of the type in 
Proposition~\ref{P:1} and \eqref{az1}.  
\end{remark}

\bigskip
\noindent
{\bf Applications beyond leaf heights.}
The results above concern only the two leaf heights $D_n$ and $L_n$, 
but the methodology can also be applied to other aspects of the random tree model.
We give two examples.

First consider $\gL_n$ := (total) length of $\CTCS(n)$.
Part of Theorem \ref{TgL} is (see \eqref{qm9})
\begin{Theorem}
\begin{align}
\label{qm91}
\E[\gL_n]  
=\frac{6}{\pi^2}n + O\bigpar{n^{-|s_1|}}.
\end{align}
\end{Theorem}

A final application concerns the probability\footnote{This relates to the number of clades of size $j$ within $\DTCS(n)$ -- see \eqref{jun12} -- and also to the study of the \emph{fringe distribution} in  \cite{beta3-arxiv}.}
$a(n,j)$ that the harmonic descent chain started at $n$ ever visits state $j$. 
 The following result was proved as  Theorem 3.1 and (6.17) of  \cite{beta3-arxiv},
 to give a first illustration of our methodology in that article.

\begin{Theorem}\label{Ta}
For every fixed $j\ge2$, 
  \begin{align}\label{e14}
  a(n,j) = \frac{6}{\pi^2}\frac{h_{j-1}}{j-1} +O\bigpar{n^{-1-|s_1|}}
.\end{align}
In particular, as \ntoo,
\begin{align}\label{e14b}
  a(n,j) \to a(j):= \frac{6}{\pi^2}\frac{h_{j-1}}{j-1}
.\end{align}
\end{Theorem}

\subsection{Outline of paper}
Section \ref{sec:exch} describes the exchangeable partitions representation  \cite{beta3-arxiv} of the limit $\CTCS(\infty)$.
Section \ref{Sprel} recalls some basic analysis surrounding Mellin transforms.
Because the finite trees are embedded within $\CTCS(\infty)$, several important expectations
for the finite tree can be expressed in terms of a specific measure $\gU$ defined for $\CTCS(\infty)$ 
(Sections \ref{Sbasic} - \ref{sec:other}).
The measure  $\gU$ is defined by its Mellin transform;
we cannot invert explicitly but do understand its behavior near zero (Section \ref{SgU}).
The remainder of the paper uses these tools to prove the Theorems stated above.
This involves classical, but rather intricate, complex analysis.

\section{The exchangeable partitions representation}
\label{sec:exch}
The relation between trees and nested families of partitions
has been used at least since \cite{haas-miermont2004}.
Its application in the context of our model is explained in detail in
\cite{beta3-arxiv}, from which the material below is taken.
See also \cite{haas-pitman} for closely related results in a greater generality.

Here we will consider $\CTCS(n)$.
We do find it convenient to adopt the biological term \emph{clade} for the set of leaves in a subtree.

Fix a level (time) $t\ge0$.
For each $n$, the clades of $\CTCS(n)$ at time $t$
define a partition $\Pi\nnn(t)$ of
$[n]:=\setn$.
Now apply a uniform random permutation of $[n]$, so now the partition $\Pi\nnn(t)$ is exchangeable.
Then, regarding  $\CTCS(n)$ as a tree on leaves $[n]$, there is a natural 
 ``delete leaf $n$  from the tree, and prune" operation \cite[Section 2.3]{beta3-arxiv}
 that yields a tree with leaves $[n-1]$.
There is a key non-obvious {\em consistency property}  
\begin{quote}
\cite[Theorem 2.3]{beta3-arxiv}
{\em
 The operation ``delete leaf $n$  from $\CTCS(n)$ and prune" gives a tree
 distributed as $\CTCS(n-1)$.
 }
 \end{quote}
So there exists a 
{\em consistent growth process}  $(\CTCS(n), n \ge 1)$ in which the partitions $\Pi\nnn(t)$ are consistent
and therefore define
a partition $\Pi(t)$ of $\bbN:=\set{1,2,\dots}$ into  clades at time $t$.
Explicitly,
$i$ and $j$ (with $i,j\in\bbN$)
are in the same part if and only if
the branchpoint separating the paths to leaves $i$ and $j$ has height $>t$,
in $\CTCS(n)$ for any $n\ge\max(i,j)$.
In the sequel we consider only this consistent and exchangeable version of $\CTCS(n)$.

Because each $\CTCS(n)$ has been made exchangeable,  $\Pi(t)$ is an exchangeable random
partition of $\bbN$, so we can exploit the theory of exchangeable
partitions. 
Denote the clades at time $t$, that is the parts of $\Pi(t)$, by
$\Pi(t)_1,\Pi(t)_2,\dots$, enumerated in order of the least elements.
In particular, the clade of leaf 1 is $\Pi(t)_1$. 
The clades $\Pi(t)_\ell$ are thus subsets of $\bbN$, and the clades of 
$\CTCS(n)$ are the sets $\Pi(t)_\ell\cap[n]$ that are non-empty.
Note that $(\Pi_t)_{t\ge0}$ determines $\CTCS(n)$ for every $n$.


Write $\left| \ \cdot \ \right|$ for cardinality.
Define, for $\ell,n\ge1$,
\begin{align}\label{d1}
  \KX^{(n)}_{t,\ell}:=\bigabs{\Pi(t)_\ell\cap[n]};
\end{align}
the sequence $\KX^{(n)}_{t,1},\KX^{(n)}_{t,2},\dots$ is thus the  sequence
of sizes of 
the clades in $\CTCS(n)$, extended by 0's to an infinite sequence.
By Kingman's fundamental result \cite[Theorem 2.1]{bertoin},
the asymptotic proportionate clade sizes,  
that is the limits
\begin{align}\label{d2}
  \XP_{t,\ell}:=\lim_\ntoo\frac{\KX^{(n)}_{t,\ell}}{n},
\end{align}
exist a.s.\ for every $\ell\ge1$, and
the random partition $\Pi(t)$ may be reconstructed (in distribution)
from the limits $(\XP_{t,\ell})_\ell$ by Kingman's paintbox construction,
stated as the following theorem.

\begin{Theorem}\label{T:paintbox}
Let $t\ge0$.
  \begin{romenumerate}
  \item \label{T:paintbox1}
If\/ $t>0$, then 
a.s.\ each $\XP_{t,\ell}\in(0,1)$, and\/ $\sum_\ell \XP_{t,\ell}=1$.
\item\label{T:paintbox2}     
Given a realization of $(\XP_{t,\ell})_\ell$,
give each integer $i\in \bbN $ a random colour $\ell$, 
with probability distribution $(\XP_{t,\ell})_\ell$,
independently for different $i$. These colours define a random partition 
$\Pi'(t)$
of
$\bbN$, which has the same distribution as $\Pi(t)$.
  \end{romenumerate}
\end{Theorem}

We note also that, as an immediate consequence of the paintbox construction,
the distribution of $\XP_{t,1}$ equals the distribution of a size-biased
sample from $\set{\XP_{t,\ell}}_{\ell=1}^\infty$
\cite[Corollary 2.4]{bertoin}.
In other words, for any function $f\ge0$ on $(0,1)$ and $t>0$,
\begin{align}  \label{pa1}
 \E [f(\XP_{t,1})] 
= \E\lrsqpar{\sum_{\ell=1}^\infty\XP_{t,\ell}f(\XP_{t,\ell})}
= \sum_{\ell=1}^\infty\E[\XP_{t,\ell}f(\XP_{t,\ell})] .
\end{align}

We have also  \cite[Theorem 4.5]{beta3-arxiv}
\begin{align}\label{b5}
\Ex [\cXt^s] 
=e^{-t(\psi(s+1)-\psi(1))},
\qquad \Re s>-1
\end{align}
which will be key to our subsequent Mellin analysis.

\section{Analysis preliminaries}\label{Sprel}

\subsection{Asymptotic expansions}\label{SSasy}

An  asymptotic expansion 
(see e.g.\ \cite[p.~724]{FS})
of a function $f(n)$ written
\begin{align}\label{asy1}
  f(n) \sim \sum_{k\ge1} \gl_k\go_k(n),
\end{align}
for some functions $\go_k(n)$ and real or complex coefficients $\gl_k$ 
($k\ge1$), 
means that $\go_{k+1}(n)=o(\go_k(n))$ as \ntoo{} for every $k\ge1$, and 
that for any $N\ge1$,
the error
\begin{align}\label{asy2}
f(n)-\sum_1^N \gl_k\go_k(n)  =O\bigpar{|\go_{N+1}(n)|}.
\end{align}
In other words,  the error when approximating
with a partial sum is of the order of the largest (non-zero) omitted term.
Note that the infinite sum $\sum_1^\infty \gl_k\go_k(n)$ does not have to
converge (and typically does not); 
this is indicated by the symbol ``$\sim$'' instead
of ``$=$'' in \eqref{asy1}.

In some expansions below 
we have several sums or a double sum in the asymptotic expansion;  
this should be interpreted as above by
rearranging the terms in decreasing order.

\subsection{The digamma function}\label{SSpsi}
The digamma function $\psi$ is defined by
\cite[5.2(i)]{NIST}
\begin{align}\label{no2}
  \psi(z):=\frac{\ddx}{\ddx z}(\log\gG(z))
=\gG'(z)/\gG(z).
\end{align}
This is a meromorphic function in the complex plane $\bbC$, with simple
poles at \set{0,-1,-2,\dots}.
We have
\cite[5.4.12]{NIST}
\begin{align}\label{no3}
  \psi(1)=-\gamma,
\qquad \psi'(1)=\zeta(2)=\tfrac{\pi^2}{6}, 
\qquad \psi''(1)=-2\zeta(3)
\end{align}
where $\gamma\doteq0.5772$ is Euler's gamma, 
and more generally for integer $n\ge1$
\cite[5.4.14]{NIST}
\begin{align}\label{no4}
  \psi(n)=h_{n-1}+\psi(1)
=h_{n-1}-\gamma
.\end{align}

We will  use the asymptotic expansion
\cite[5.11.2]{NIST}
(obtained by logarithmic differentiation of Stirling's formula), 
where $B_k$ denotes the Bernoulli  numbers,
\begin{align}
  \label{5.11.2}
\psi(z)\sim \log z -\frac{1}{2}z\qw -\sumk \frac{B_{2k}}{2k}z^{-2k}
\end{align}
as $|z|\to\infty$ in any fixed sector $|\arg(z)|\le \pi-\gd<\pi$.
Note that, since it is valid in sectors, \eqref{5.11.2} may be termwise
differentiated to yield asymptotic expansions of $\psi'(z)$ and higher
derivatives, see \cite[5.15.8--9]{NIST}.
We have by \eqref{no4} and \eqref{5.11.2}
the asymptotic expansion
\begin{align}\label{an1}
    h_{n-1}
\sim \log n+\gamma -\frac{1}{2}n\qw -\sumk \frac{B_{2k}}{2k}n^{-2k},
\end{align}
and in particular the well-known
expansion
\begin{align}\label{au79}
  h_{n-1}=\log n +\gamma - \frac{1}{2n} + O\bigpar{n^{-2}}.
\end{align}

As the Mellin formula at \eqref{m1} will show, the roots of $\psi(s)=\psi(1)$ will play an important role in our
proofs and results, so we first describe them.

\begin{Lemma}\label{Lpsi}
  The roots of the equation $\psi(s)=\psi(1)$ 
are all real and can be enumerated in
  decreasing order as $s_0=1>s_1>s_2>\dots$, with $s_i\in(-i,-(i-1))$ for
  $i\ge1$. 
Numerically, 
$s_1\doteq-0.567$  
and 
$s_2\doteq-1.628$.  

More generally, for any $a\in\bbR$, the roots of $\psi(s)=a$ are all real
and can be enumerated 
as $s_0(a)=1>s_1(a)>s_2(a)>\dots$, with $s_i(a)\in(-i,-(i-1))$ for
  $i\ge1$. 
\end{Lemma}
We let $s_i$ have this meaning throughout the paper.
See \cite[\S5.4(iii)]{NIST} for $s_i(a)$ in the case $a=0$.

\begin{proof}
Recall that $\psi(s)$ is a meromorphic function of $s$, with poles at
$0,-1,-2,\dots$. For any other complex $s$ we have the standard formulas
\cite[5.7.6 and 5.15.1]{NIST}
\begin{align}\label{psi}
  \psi(s)&=-\gamma+\sumko\Bigpar{\frac{1}{k+1}-\frac{1}{k+s}}
,\\\label{psi'}
\psi'(s)&=\sumko\frac{1}{(k+s)^2}
.\end{align}

If $\Im s>0$, then $\Im(1/(k+s))<0$ for all $k$ and thus \eqref{psi}
  implies
$\Im\psi(s)>0$.
Similarly, if $\Im s<0$, then $\Im\psi(s)<0$.
Consequently, all roots of $\psi(t)=a\in\bbR$ are real.

For real $s$, \eqref{psi'} shows that $\psi'(s)>0$. 
We can write $\bbR\setminus\set{\text{the poles}}=
\bigcup_{i=0}^\infty I_i$ with 
$I_0:=(0,\infty)$ and $I_i:=(-i,-(i-1))$ for $i\ge1$;
it then follows that $\psi(s)$ is strictly increasing in each interval
$I_i$.
Moreover, by \eqref{psi} (or general principles),
at the poles we have the limits $\psi(-i-0)=+\infty$ and
$\psi(-i+0)=-\infty$ ($i\ge0$),
and furthermore (by \eqref{5.11.2})
$\psi(s)\to\infty$ as $s\upto+\infty$, so $\psi$ maps each interval $I_i$ to $(-\infty,\infty)$.
Consequently, $\psi(s)=a$
has exactly one root $s_i(a)$ in each $I_i$.
(See also the graph of $\psi(s)$ in \cite[Figure 5.3.3]{NIST}.)
In the special case $a=\psi(1)$, obviously the positive root is
$s_0=1$. 

The numerical values are obtained by \texttt{Maple}.
\end{proof}

\subsection{Mellin transforms}
If $f$ is a function defined on $\bbR_+:=\qooo$, then 
its \emph{Mellin transform}
is defined by
\begin{align}\label{mellin}
  \M f(s):=\intoo f(x) x^{s-1}\dd x
\end{align}
for all complex $s$ such that the integral converges absolutely.
It is well known  that the domain of all such $s$ is a strip 
$\cD=\set{s:\Re s\in J_f}$ 
for some interval $J_f\subseteq\bbR$ (possibly empty or degenerate),
and that $\M f(s)$ is analytic in the interior $\cDo$ of $\cD$
(provided $\cDo$ is non-empty, i.e.,  $J_f$ is neither empty nor degenerate).

In analogy to \eqref{mellin},
if $\mu$ is a (possibly complex) measure on $\bbR_+$, then 
its \emph{Mellin transform}
is defined by
\begin{align}\label{mellin2}
  \M\mu(s):=\intoo  x^{s-1}\dd\mu(x)
\end{align}
for all complex $s$ such that the integral converges absolutely.
Again, the domain of all such $s$ is a strip 
$\cD=\set{s:\Re s\in J_\mu}$ 
for some interval $J_\mu\subseteq\bbR$ (possibly empty or degenerate),
and  $\M\mu(s)$ is analytic in the interior $\cDo$ of $\cD$.
Note that the Mellin transform \eqref{mellin2} for real $s$ is just the
moments of the measure $\mu$, with a simple shift of the argument.

\begin{remark}
There are in the literature 
also other definitions of the Mellin transform of measures;
see for example the Mellin-Stieltjes transform in \cite[Appendix~D]{DrmotaSzp}.
We use the definition above, which is convenient when we want to identify an
absolutely continuous measure with its density function.
\end{remark}

\section{The central idea}\label{Sbasic}
Recall the limit 
$ \XP_{t,1}:=\lim_\ntoo\frac{\KX^{(n)}_{t,1}}{n} $
from \eqref{d2}.
The central idea in this article is to define an infinite measure $\gU$ on
$\oio$ by 
\begin{align}\label{e7}
  \gU :=\intoo \cL(\XP_{t,1})\dd t.
\end{align}
This means that for any (measurable) function $f\ge0$ on $(0,1)$,
\begin{align}\label{cx}
  \intoi f(x)\dd\gU(x) 
=\intoo \intoi f(x) \dd\cL(\XP_{t,1})(x)\dd t
=\intoo \E [f(\XP_{t,1})]\dd t.
\end{align}
This extends by linearity to every  complex-valued $f$ such that
$\intoo \E[| f(\XP_{t,1})|]\dd t<\infty$.
The identity \eqref{cx} enables us to express several important
expectations in terms of the measure $\gU$: these are listed 
in Proposition \ref{P:1} below.
As well as $\E[D_n]$ and $\E[L_n]$, these involve
the total length $\gL_n$ of $\CTCS(n)$, and
 the ``occupation probability", that is
 \begin{equation}
 \mbox{
 $a(n,i) := $ probability that the harmonic descent chain started at state $n$ is ever in state $i$.
 }
 \label{def:ani}
 \end{equation}
 So $a(n,n) = a(n,1) = 1$.

To illustrate the methodology we derive the first of these identities below
(the others will be proved in Section \ref{sec:other}).

\begin{Proposition}
\label{P:1}
For any $n\ge1$:
\begin{align}\label{cxED}
\E[D_n] 
=\intoi \bigpar{1-(1-x)^{n-1}}\dd \gU(x)
.\end{align}

\begin{align}\label{cy2}
a(n,j)
=h_{j-1} \binom {n-1}{j-1} \intoi x^{j-1}(1-x)^{n-j}\dd\gU(x),
\quad 2\le j\le n
.\end{align}

\begin{align}\label{cy3}
  \E[ L_n] = \sum_{j=2}^na(n,j)
=\intoi\sum_{j=2}^n h_{j-1}\binom {n-1}{j-1} x^{j-1}(1-x)^{n-j}\dd\gU(x)
.\end{align}

\begin{align}\label{qm1}
  \E[\gL_n]& 
= \intoi \frac{1}{x}\Bigpar{1-(1-x)^n-nx(1-x)^{n-1}}\dd\gU(x)
. \end{align}

\end{Proposition}

But what is the measure $\gU$, explicitly?
We first note that \eqref{cx} and \eqref{b5} tell  us the 
Mellin transform of the 
measure $\gU$:  
\begin{align}
\label{m1}
\M\gU(s):=
  \intoi x^{s-1}\dd\gU(x)
=\intoo \Ex \bigsqpar{\XP_{t,1}^{s-1}}\dd t
= \frac{1}{\psi(s)-\psi(1)},
\qquad \Re s>1.
\end{align}
The measure $\gU$ is determined by its  Mellin transform.
We do not know how to invert the transform \eqref{m1} to obtain a useful explicit formula for
$\gU$, but what is relevant to our asymptotics is mainly the behavior of $\gU$
near $0$, which will be given in the ``inversion estimate", \refL{LM}.

\subsection{An illustrative identity}
To prove \eqref{cxED},
consider $\CTCS(n)$ for an integer $n\ge1$, and recall that the clades
at time $t$ are precisely the blocks of the partition $\Pi(t)\cap[n]$.
Fix a non-empty subset  $J\subseteq[n]$ of size $j=|J|\ge1$.
If we use the paintbox construction in Theorem \ref{T:paintbox}
to reconstruct a copy $\Pi'(t)$ of $\Pi(t)$, 
we see that
\begin{quote}
conditioned on $(P_{t,\ell})_{\ell=1}^\infty$,
the probability that $J$ is a block in $\Pi'(t)\cap[n]$
equals $\sum_\ell \XP_{t,\ell}^j(1-\XP_{t,\ell})^{n-j}$. 
\end{quote}
We take the expectation;
$\Pi'(t)$ has the same distribution as $\Pi(t)$, 
so the distinction between them then disappears 
and we obtain that
for any fixed non-empty set 
$J\subseteq[n]$ of size $|J|=j$ and $t>0$,
\begin{align}\label{cx1}
\Pr\bigpar{J \text{ is a clade at time }t}&
=
\Ex \Bigsqpar{\sum_\ell \XP_{t,\ell}^j(1-\XP_{t,\ell})^{n-j}}
.\end{align}
By \eqref{pa1}, we can rewrite \eqref{cx1} as
\begin{align}\label{cx2}
\Pr\bigpar{J \text{ is a clade at time }t}
=
  \Ex \bigsqpar{ \XP_{t,1}^{j-1}(1-\XP_{t,1})^{n-j}}
.\end{align}

In particular, since $D_n\le t$ exactly when $\set1$ is a clade at time $t$,
\eqref{cx2} with $J=\set1$ and thus $j=1$ yields
\begin{align}\label{cx3}
  \Pr(D_n>t) = 1-\Pr(D_n\le t) 
= 1-  \Ex \bigsqpar{(1-\XP_{t,1})^{n-1}}
= \Ex \bigsqpar{1-(1-\XP_{t,1})^{n-1}}
.\end{align}
Consequently,
\begin{align}\label{cx4}
\E[D_n] = \intoo \Pr(D_n>t) \dd t
=\intoo \Ex \bigsqpar{1-(1-\XP_{t,1})^{n-1}}\dd t
,\end{align}
which by \eqref{cx} yields our desired formula \eqref{cxED}.

As with the other identities in Proposition \ref{P:1},
it is then
straightforward (but sometimes intricate) classical analysis to combine \eqref{cxED} with the inversion estimate (\refL{LM}) to obtain
an asymptotic expansion of $\E[D_n]$: see \refSs{SED1},   \ref{SEL1}, and \ref{SgL}.

\begin{Remark} We give in \refS{SED2} a different argument, \emph{not} using
\refL{LM}, where we obtain the same asymptotic expansion 
by directly combining \eqref{cxED} and the Mellin transform
\eqref{m1}, using a version of Parseval's formula.
This method is less intutive, but leads to exact formulas involving
line integrals in the complex plane, which through residue calculus
yield another proof of the asymptotic expansion.
\end{Remark}

\begin{Remark}\label{RED1}
We remark that \eqref{cxED} combined with \eqref{m1} 
also yields an exact formula for $\Ex[D_n]$.
By the binomial theorem we obtain, using \eqref{no4}, 
\begin{align}\label{au5}
  \Ex[D_n]& 
=\intoi \sum_{j=1}^{n-1}(-1)^{j-1}\binom{n-1}{j}x^{j}\dd\gU(x)
\\\notag&
= \sum_{j=1}^{n-1}(-1)^{j-1}\binom{n-1}{j}\frac{1}{\psi(j+1)-\psi(1)}
= \sum_{j=1}^{n-1}(-1)^{j-1}\binom{n-1}{j}\frac{1}{h_j}.
\end{align}
However, since this is an alternating sum, it does not seem easy to derive
asymptotics from it, and we will not use it. 
\end{Remark}

\section{Other identities involving $\gU$}
\label{sec:other}

\subsection{The occupation probability}\label{SSOP}


As noted in Section \ref{sec:Lh}, if we follow the path from the root to leaf 1
(or equivalently to a uniformly random leaf)
in $\CTCS(n)$ or $\DTCS(n)$, then the sequence of clade sizes follows a
Markov chain, which we call the
{\em harmonic descent} (HD) chain, see \cite{beta2-arxiv} and \cite{HDchain}.
In particular, we are interested in the ``occupation probability", that is
 \begin{equation} \label{def:ani2}
 a(n,i) := 
\text{probability that the HD chain started at state $n$ is ever in state $i$}.
 \end{equation}
In other words, $a(n,i)$ is the probability that a given (or random) leaf
in $\CTCS(n)$ or $\DTCS(n)$ belongs to some subtree with exactly $j$ leaves. 
Writing $N_n(j)$ for the number of such subtrees, we clearly have 
 \begin{equation}\label{jun12}
\E[N_n(j)]
= n a(n,j)/j .
 \end{equation}

Similar to the argument above that proved  \eqref{cxED}, one can prove 
\cite[(6.7), (6.6), and (6.10)]{beta3-arxiv}
the identity \eqref{cy2}, which we repeat as
\begin{align}\label{cy22}
a(n,j)
=h_{j-1} \binom {n-1}{j-1} \intoi x^{j-1}(1-x)^{n-j}\dd\gU(x), \ 2 \le j \le n
.\end{align}

The methodology via Lemma \ref{LM} that we use later was already applied in \cite{beta3-arxiv} to establish the asymptotics of $a(n,j)$ stated in Theorem \ref{Ta},
so we do not repeat it here.

\begin{Remark}\label{R:anj}
Similarly to \eqref{au5}, we obtain from \eqref{cy22}, \eqref{m1}, and
\eqref{no4} the exact  formula,
using a binomial expansion and manipulation of binomial coefficients,
\begin{align}\label{cec1}
  a(n,j) 
&
=h_{j-1} \binom {n-1}{j-1} \intoi \sum_{k=0}^{n-j} 
  \binom{n-j}{k}x^{j-1}(-x)^{k}\dd\gU(x)
\\\notag&
=h_{j-1}\sum_{k=0}^{n-j}\binom {n-1}{j-1}\binom{n-j}{k} 
  \frac{(-1)^k}{\psi(j+k)-\psi(1)}
\\\notag&
=\sum_{k=0}^{n-j}(-1)^k\binom {n-1}{j-1,k,n-j-k}\frac{h_{j-1}}{h_{j+k-1}}
.\end{align}
\end{Remark}

\subsection{The expected hop-height $\E[L_n]$}\label{SSLn}

 In $\CTCS(n)$, the expected lifetime of a clade of size
$j$ is $1/h_{j-1}$, and hence \eqref{cy22} implies
\begin{align}\label{cy2.5}
  \E [D_n] = \sum_{j=2}^n \frac{1}{h_{j-1}}a(n,j)
=\intoi\sum_{j=2}^n \binom {n-1}{j-1} x^{j-1}(1-x)^{n-j}\dd\gU(x),
\end{align}
which by summing the binomial series yields another proof of \eqref{cxED}.
Similarly, for $\DTCS(n)$, the height $L_n$ equals the number of clades of
sizes $\ge2$ that leaf 1 ever belongs to, and thus we have
\begin{align}\label{cy32}
  \E [L_n] = \sum_{j=2}^na(n,j)
=\intoi\sum_{j=2}^n h_{j-1}\binom {n-1}{j-1} x^{j-1}(1-x)^{n-j}\dd\gU(x),
\end{align}
We will proceed to the asymptotic expansion of $  \E [L_n]$  in \refS{SEL1}.

\subsection{The expected length $\E[\gL_n]$}\label{SSgLn}

The number of edges of $\CTCS(n)$ equals $n-1$.
Identifying {\em length} of an edge with {\em duration of time}, one can consider the length $\Lambda_n$ of 
$\CTCS(n)$, that is the sum of all edge-lengths.
In other words, $\gL_n$ is the sum of the lifetimes of all internal nodes. 
Since
there are $N_n(j)$ nodes with $j$ descendants and each of them has an
expected lifetime of $1/h_{j-1}$, with
the lifetimes  independent of $N_n(j)$, 
we obtain,
using \eqref{jun12} and \eqref{cy22}, 
\begin{align}\label{qm11}
  \E[\gL_n]& 
= \sum_{j=2}^n \frac{1}{h_{j-1}}\E [N_n(j)] 
= \sum_{j=2}^n\binom nj \intoi x^{j-1}(1-x)^{n-j}\dd\gU(x)
\\\notag&
= \intoi \frac{1}{x}\Bigpar{1-(1-x)^n-nx(1-x)^{n-1}}\dd\gU(x).
\end{align}

We will proceed to the asymptotic expansion of $  \E [\gL_n]$  in \refS{SgL}.

\section{The measure $\gU$}\label{SgU}


The following {\em inversion estimate} will allow us to pass from the identities in Proposition \ref{P:1} to sharp $n \to \infty$ asymptotics for expectations.
Part of this lemma is also given in \cite{beta3-arxiv};
for completeness, we give the entire proof.

\begin{Lemma}\label{LM}
  Let $\gU$ be the infinite
measure on $\oio$ having the Mellin transform
\eqref{m1}.
Then 
$\gU$ is absolutely continuous, with a continuous density $\gu(x)$ on
$(0,1)$.
Furthermore:
\begin{romenumerate}
  
\item 
The density $\gu(x)$ satisfies
\begin{align}\label{lmb}
\gu(x)
= \frac{6}{\pi^2x}+O\bigpar{x^{-s_1}+x^{-s_1}|\log x|\qw},
\end{align}
uniformly for $x\in(0,1)$, 
where $s_1\doteq -0.567$ is the largest negative root of $\psi(s)=\psi(1)$.
In other words,
for $x\in(0,1)$, 
\begin{align}
  \label{lmr}
\gu(x)=\frac{6}{\pi^2}x\qw + r(x)
\end{align}
where
\begin{align}\label{aur}
r(x)= O\bigpar{x^{-s_1}+x^{-s_1}|\log x|\qw}  
\end{align}
and thus, in particular,
for $x\in(0,\frac12)$ say,
\begin{align}\label{aur1}
r(x)=O\bigpar{x^{|s_1|}}
.\end{align}
Furthermore, 
$\gU(\gd,1)<\infty$ 
for every $\gd>0$
and
\begin{align}
  \label{aur2}
\intoi |r(x)|\dd x<\infty. 
\end{align}

\item 
More generally, 
with $0>s_1>s_2>\dots$ denoting the negative roots of $\psi(s)=\psi(1)$,
for any $N\ge0$, 
\begin{align}\label{lmrN}
\gu(x)
= \frac{6}{\pi^2}x\qw+
\sum_{i=1}^N \frac{1}{\psi'(s_i)} x^{|s_i|}
+r_N(x),
\qquad 0<x<1,
\end{align}
where
\begin{align}\label{aurN}
r_N(x)=O\bigpar{x^{|s_{N+1}|}(1+|\log x|^{-1})}
\end{align}
and 
\begin{align}
  \label{aur2N}
\intoi |r_N(x)|\dd x<\infty. 
\end{align}
\end{romenumerate}
\end{Lemma}

By \eqref{aur1}--\eqref{aur2}, the Mellin transform $\Mr(s)$ exists for 
$\Re s>s_1$, and we obtain by \eqref{lmr} and \eqref{m1}
\begin{align}
  \label{mr}
\Mr(s) = \MgU(s)-\frac{6}{\pi^2}\intoi x^{s-2}\dd x
= \frac{1}{\psi(s)-\psi(1)}-\frac{6}{\pi^2}\cdot\frac{1}{s-1},
\end{align}
first for $\Re s>1$ and then by analytic continuation for $\Re s>s_1$;
note that the \rhs{} of \eqref{mr} has a removable singularity at $s=1$
since the residues of the two terms cancel.

\begin{proof}[Proof of \refL{LM}]
We begin by noting, 
as  in the proof of Lemma \ref{Lpsi}, that
the Mellin transform $1/\bigpar{\psi(s)-\psi(1)}$
in \eqref{m1}  extends to 
a meromorphic function in the entire
complex plane, whose poles are the roots of
$  \psi(s)=\psi(1)$.
As shown in Lemma \ref{Lpsi}, besides the obvious pole $s_0=1$,
the other poles are real and negative, and thus can be ordered
$0>s_1>s_2>\dots$. In particular,
there are no other poles than 1 in the half-plane
$\Re s>s_1$, with $s_1\doteq -0.567$.
The residue at the pole $s_0=1$ is,
using \eqref{no3},
\begin{align}\label{m3}
  \Res_{s=1}\frac{1}{\psi(s)-\psi(1)}
=\frac{1}{\psi'(1)}=\frac{6}{\pi^2}
.\end{align}

We cannot immediately use standard results on Mellin inversion
(as  in \cite[Theorem 2(i)]{FGD})
because the Mellin transform in \eqref{m1} decreases too slowly 
as $\Im s\to\pm\infty$ to be integrable on a vertical line $\Re s=c$.%
\footnote{And we cannot use
\cite[Theorem 2(ii)]{FGD} since we do not know that $\gU$ has a density
that is locally of bounded variation.}
In fact, \eqref{5.11.2} implies that
\begin{align}\label{m4}
  \psi(s) = \log s + o(1) = \log|s| + O(1) = \log|\Im s|+O(1)
\end{align}
as $\Im s\to\infty$ with $s$ in, for example, any half-plane $\Re s\ge c$.

We overcome this problem by differentiating the Mellin transform, but we
first subtract the leading term corresponding to the pole at 1.
Since $\gU$ is an infinite measure, we first replace it by 
$\nu$ defined by 
$\dd\nu(x)=x\dd\gU(x)$; note that $\nu$ is also a measure on $\oio$, and 
taking $s=2$ in \eqref{m1} shows that $\nu$ is a finite measure.

Next, define $\nu_0$ as the measure $(6/\pi^2)\dd x$ on $\oio$, and let $\nux$
be the (finite) signed measure $\nu-\nu_0$. 
Then $\nux$ has the Mellin transform, by \eqref{m1},
\begin{align}\label{m5}
\tnux(s)
&:=
  \intoi x^{s-1}\dd\nux(x)
=  \intoi x^{s}\dd\gU(x)
- \frac{6}{\pi^2} \intoi x^{s-1}\dd x
\\\notag&\phantom:
=\frac{1}{\psi(s+1)-\psi(1)}-\frac{6}{\pi^2s},
\qquad \Re s>0.
\end{align}
We may here differentiate under the integral sign, which gives
\begin{align}\label{m6a}
\tnux'(s)
&:=
  \intoi(\log x) x^{s-1}\dd\nux(x)
\\ \label{m6b}&\phantom:
=
-\frac{\psi'(s+1)}{(\psi(s+1)-\psi(1))^2}
+\frac{6}{\pi^2s^2},
\qquad \Re s>0.
\end{align}
The Mellin transform $\tnux(s)$ in \eqref{m5} extends to a meromorphic
function in $\bbC$ with
(simple) poles $(s_i-1)_1^\infty$; note that there is no pole at $s_0-1=0$,
since the residues there of the two terms in \eqref{m5} cancel by \eqref{m3}.
Furthermore, the formula \eqref{m6b} for $\tnux'(s)$ then holds
for all $s$ (although the integral in \eqref{m6a} diverges unless $\Re s>0$).

For any real $c$ we have, on the vertical line $\Re s=c$, as $\Im s\to\pm\infty$,
that $\psi(s)\sim \log|s|$ by \eqref{m4}, and also, by differentiation of
\eqref{m4} (see \cite[5.15.8]{NIST}) 
that $\psi'(s)\sim s\qw$.
It follows from \eqref{m6b} that 
\begin{align}\label{m6c}
\tnux'(s)=O(|s|\qw\log\qww|s|)  
\end{align}
on the
line $\Re s=c$,  
for $|\Im s|\ge2$ say,
and thus
$\tnux'$ is integrable on this line
unless $c$ is one of the poles $s_i-1$.
In particular, taking $c=1$ and thus $s=1+u\ii$ ($u\in\bbR$),
we see that the function 
\begin{align}\label{mma}
\tnux'(1+\ii u)=\intoi x^{\ii u}\log (x) \dd\nux(x)  
\end{align}
is integrable. 
The change of variables $x=e^{-y}$ shows that the function
\eqref{mma} 
is the Fourier transform of the signed measure 
on $\bbR_+$ that corresponds to $\log(x)\dd\nux(x)$.
This measure 
on $\bbR_+$
is thus a finite signed measure with integrable Fourier transform,
which implies that it is absolutely continuous with a continuous density.
Reversing the change of variables, we thus see that the signed measure
$\log(x)\dd\nux(x)$ is absolutely continuous with a continuous density on
$(0,1)$. Moreover, denoting this density by $h(x)$, we obtain the standard
inversion formula for the Mellin transform
\cite[Theorem 2(i)]{FGD}, \cite[1.14.35]{NIST}:
\begin{align}\label{m8}
 h(x)=\frac{1}{2\pi\ii}\int_{c-\infty\ii}^{c+\infty\ii} x^{-s}\tnux'(s)\dd s,
\qquad x>0,
\end{align}
with $c=1$.
Furthermore, the integrand in \eqref{m8} is analytic in the half-plane $\Re
s>s_1-1$, and the estimate \eqref{m6c} above is
uniform for $\Re s$ in any compact interval and $|\Im s|\ge2$. 
Consequently, we may shift the line of integration in \eqref{m8} to any
$c>s_1-1$. Taking absolute values in \eqref{m8}, and recalling that
$\tnux'(s)$ is integrable on the line, then yields
\begin{align}\label{m9}
  h(x) = O\bigpar{x^{-c}}
\end{align}
for any $c>s_1-1$.

Reversing the transformations above, we see that $\nux$ has the density
$(\log x)\qw h(x)$, and thus $\nu$ has the density
$(\log x)\qw h(x)+ 6/\pi^2$, and, finally, that $\gU$ has the density
\begin{align}\label{m10}
\gu(x):=\frac{\ddx\gU}{\ddx x}
=\frac{1}{x}\frac{\ddx\nu}{\ddx x}
= \frac{6}{\pi^2x}+\frac{1}{x\log x} h(x),
\qquad 0<x<1.
\end{align}
Furthermore, \eqref{m10} and \eqref{m9} have the form of the claimed
estimate \eqref{lmb},
although with the weaker error term $O(x^{-s_1-\eps}|\log x|\qw)$ 
for any $\eps>0$.

To obtain the claimed error term, we note that the residue of
$\tnux(s)$
at $s_1-1$ is $a_1:=1/\psi'(s_1)$. Let $\nu_1$ be the measure
$a_1x^{1-s_1}\dd x$ on $\oio$; then $\nu_1$ has Mellin transform
\begin{align}\label{m11}
  \widetilde{\nu_1}(s)=a_1\intoi x^{s-1}x^{1-s_1}\dd x
=\frac{a_1}{s+1-s_1},
\qquad \Re s> s_1-1.
\end{align}
It follows from \eqref{m5} and \eqref{m11}
that the signed measure $\nu-\nu_0-\nu_1=\nux-\nu_1$ has the Mellin
transform 
\begin{align}\label{m12}
\frac{1}{\psi(s+1)-\psi(1)}-\frac{6}{\pi^2s}-\frac{a_1}{s+1-s_1},
\end{align}
which is an analytic function in the half plane $\Re s > s_2-1$.
Hence, the same argument as above yields the estimate
\begin{align}\label{lm1}
\gu(x)
= \frac{6}{\pi^2x}+
\frac{1}{\psi'(s_1)} x^{-s_1}
+O\bigpar{x^{-s_2-\eps}|\log x|\qw},
\qquad x\downarrow 0,
\end{align}
for any $\eps>0$, which in particular yields \eqref{lmb},
and thus \eqref{lmr}--\eqref{aur}.

We may continue the argument above further and subtract similar terms for
any number of poles; this leads to the estimate, for any $N\ge1$,
\begin{align}\label{lm++}
\gu(x)
= \frac{6}{\pi^2x}+
\sum_{i=1}^N \frac{1}{\psi'(s_i)} x^{-s_i}
+O\bigpar{x^{-s_{N+1}}(1+|\log x|^{-1})},
\qquad 0<x<1,
\end{align}
and thus \eqref{lmrN}--\eqref{aurN}.

Finally, $\gU(\gd,1)<\infty$ for $\gd>0$ follows directly from \eqref{m1}
with $s=2$, say. Hence, \eqref{lmr} implies 
$\int_{\xfrac12}^1|r(x)|\dd x<\infty$, while 
$\int_0^{\xfrac12}|r(x)|\dd x<\infty$ follows from \eqref{aur1};
thus \eqref{aur2} holds. The same argument yields \eqref{aur2N}.
\end{proof}

\begin{Remark}\label{RLM1-}
  Although the function $h(x)$ is continuous also at $x=1$ (and thus  $h(1)=0$),
the density $\gu(x)$ diverges as $x\upto1$ because of
the factor $\log x$ in the denominator in \eqref{m10};
in fact, it can be shown by similar arguments that
\begin{align}\label{lm1a}
  \gu(1-y) \sim \frac{1}{y\,|\log y|^2},
\qquad y\downto0.
\end{align}
Hence, 
$r(x)$ and $r_N(x)$ are unbounded on $(0,1)$ and
the error terms in \eqref{lmb}, \eqref{aur},  and \eqref{aurN}
cannot be simplified as in \eqref{aur1}
on the entire interval $(0,1)$.
\end{Remark}

\section{The expected leaf height $\E[D_n]$}\label{SED1}

Armed with \refL{LM}, we may now, as said in \refS{Sbasic},
obtain an asymptotic expansion of $\E[D_n]$ from \eqref{cxED},
which extends the first terms given
in \cite[Theorem 1.1]{beta1}.
We begin by finding the leading terms, using \eqref{lmb}--\eqref{aur} only: that gives Proposition \ref{P:2}.
We then extend this to a full asymptotic expansion.

\subsection{Leading terms}
We denote the integrand in \eqref{cxED} by
\begin{align}\label{fn}
  f_n(x):=
  \begin{cases}
{1-(1-x)^{n-1}}, & 0<x<1,
\\
0, & x\ge1.
  \end{cases}
\end{align}
In \eqref{cxED}, we also recall that $\dd\gU(x)=\gu(x)\dd x$ 
and substitute $\gu(x)$ 
using \eqref{lmr};
this yields two terms:
\begin{align}
  \label{hw11}
\E[D_n]= \intoi f_n(x) \frac{6}{\pi^2x}\dd x + \intoi f_n(x) r(x)\dd x
.\end{align}
For the first (main) term we 
use the following lemma (stated in a  general form for later use),
which gives the Mellin transform of  $f_n$.

\begin{Lemma}\label{L99}
Fix $n\ge1$.  The Mellin transform
  \begin{align}\label{su1}
    \Mfn(s):=\intoi x^{s-1}\bigpar{1-(1-x)^{n-1}}\dd x,
\qquad \Re s>-1,
  \end{align}
is analytic in the half-plane $\Re s>-1$ and is given explicitly by
  \begin{align}\label{su2}
    \Mfn(s)&=
\frac{1}{s}-\frac{\gG(s)\gG(n)}{\gG(n+s)}
= \frac{1}{s}\Bigpar{1-\frac{\gG(s+1)\gG(n)}{\gG(n+s)}},
\qquad s\neq0,
\\\label{su3}
\Mfn(0)&=h_{n-1}.
  \end{align}
\end{Lemma}

\begin{proof}
  It is elementary that 
the integral in \eqref{su1} converges and defines an analytic (in fact,
rational) function for $\Re s>-1$.
For $\Re s>0$, we may write 
$\Mfn(s)=\intoi x^{s-1} \dd x - \intoi x^{s-1}(1-x)^{n-1}\dd x$,
and \eqref{su2} follows by evaluating the beta integral.
Hence, \eqref{su2} follows by analytic continuation.

For $s=0$, the right-hand side of \eqref{su2} has a removable singularity,
with the value, by the definition of derivative (or L'H{\^o}pital's rule),
$\gG'(s)=\psi(s)\gG(s)$, and \eqref{no4},
\begin{align}\label{su4}
  \Mfn(0)
= -\frac{\ddx}{\dd s}\frac{\gG(s+1)\gG(n)}{\gG(n+s)}\Bigr|_{s=0}
=-\psi(1)+\psi(n)
=h_{n-1},
\end{align}
which shows \eqref{su3}.
\end{proof}

By \eqref{cxED}, \eqref{lmr}, and \refL{L99},
the main term of $\E[D_n]$ in \eqref{hw11} is
\begin{align}  \label{au6}
\intoi&\bigsqpar{1-(1-x)^{n-1}}\frac{6}{\pi^2x}\dd x
= \frac{6}{\pi^2}\Mfn(0)=\frac{6}{\pi^2}h_{n-1}.
\end{align}

For the remainder term in \eqref{hw11}, we use \eqref{aur1} and \eqref{aur2}.
Consequently, 
\begin{align}\label{au7}
\intoi\bigsqpar{1&-(1-x)^{n-1}}r(x)\dd x
\\\notag&
=\intoi r(x)\dd x
-\int_0^{1/2}(1-x)^{n-1}r(x)\dd x 
-\int_{1/2}^1(1-x)^{n-1}r(x)\dd x 
\\\notag&
=\Mr(1)
+\int_0^{1/2}(1-x)^{n-1}O(x^{|s_1|})\dd x 
+\int_{1/2}^1O\bigpar{2^{-n}|r(x)|}\dd x 
\\\notag&
=\Mr(1)+O(n^{-1-|s_1|})+O(2^{-n})
=\Mr(1)+O(n^{-1-|s_1|}),
\end{align}
where we estimate the penultimate integral using $(1-x)^{n-1}\le e^{-(n-1)x}$
(or by another beta integral).
Combining \eqref{cxED}, \eqref{lmr}, \eqref{au6} and \eqref{au7} we obtain
\begin{align}\label{au8}
  \Ex[D_n] = \frac{6}{\pi^2}h_{n-1} + \Mr(1) +O(n^{-1-|s_1|}),
\end{align}
which by \eqref{au79} yields the desired leading terms:
\begin{Proposition}
\label{P:2}
\begin{align}\label{au9}
  \Ex[D_n] = \frac{6}{\pi^2}\log n + c_0 + c_{-1}n\qw+ O(n^{-1-|s_1|})
\end{align}
for  $c_0:=\Mr(1)+\frac{6}{\pi^2}\gamma$ and $c_{-1}=-\frac3{\pi^2}
$.
\end{Proposition}
In particular, this verifies the conjectured formula \cite[(2.17)]{beta1}.
Moreover, we may compute the constants $\Mr(1)$ 
above by taking the limit as $s\to1$ in \eqref{mr}.
A simple calculation (using a Taylor expansion of $\psi(1+\eps)-\psi(1)$)
yields
\begin{align}\label{aua3}
\Mr(1)
=-\frac{\psi''(1)}{2(\psi'(1))^2}
=\frac{\zeta(3)}{\zeta(2)^2},
\end{align}
using $\psi'(1)=\zeta(2)$ and $\psi''(1)=-2\zeta(3)$,
see \eqref{no3} and \eqref{psi'} or \cite[5.7.4]{NIST}.
Hence,
\begin{align}\label{aua4}
  c_0 = \frac{\zeta(3)}{\zeta(2)^2}+\frac{\gamma}{\zeta(2)}
\doteq  0.795155660439
\end{align}
which agrees to 10 decimals with the numerical estimate in \cite{beta1}.

\subsection{A full asymptotic expansion}
\label{sec:EDfull}
The expansion \eqref{au8} is easily extended to a full asymptotic expansion
of $\E[D_n]$, stated in the introduction as Theorem \ref{TD1}.
\begin{Theorem}\label{TD}
  We have the asymptotic expansion
  \begin{align}\label{aw1}
  \Ex[D_n] \sim 
\frac{6}{\pi^2}h_{n-1} 
+ \frac{\zeta(3)}{\zeta(2)^2}
-\sum_{i=1}^\infty \frac{\gG(|s_i|+1)}{\psi'(s_i)}\frac{\gG(n)}{\gG(n+|s_i|+1)}
\end{align}
where $0>s_1>s_2>\dots$ are the negative roots of $\psi(s)=\psi(1)$.
Alternatively, we have
  \begin{align}\label{aw2}
  \Ex[D_n] \sim 
\frac{6}{\pi^2}\log n
+\sum_{i=0}^\infty c_{i} n^{-i}
+\sum_{j=1}^\infty\sum_{k=1}^\infty c_{j,k}\,n^{-|s_j|-k}
\end{align}
for some coefficients $c_i$ and $c_{j,k}$ that can be found explicitly;
in particular, $c_0$ is given by \eqref{aua4} and $c_1=-3/\pi^2$.
\end{Theorem}

\begin{proof}
We use again \eqref{hw11}, but we now note also that \eqref{lmr} and
\eqref{lmrN} yield
\begin{align}\label{av3}
r(x)=
\sum_{i=1}^N \frac{1}{\psi'(s_i)} x^{|s_i|}
+ r_N(x),
\qquad 0<x<1.
\end{align}
Hence, similarly to \eqref{au7} but now using also \eqref{aurN} and
\eqref{aur2N} and evaluating beta integrals,
\begin{align}\label{av5}
&\intoi\bigsqpar{1-(1-x)^{n-1}}r(x)\dd x
\\\notag&
=\intoi r(x)\dd x
-\sum_{i=1}^N \frac{1}{\psi'(s_i)} \intoi x^{|s_i|}(1-x)^{n-1}\dd x
-\intoi(1-x)^{n-1}r_N(x)\dd x 
\\\notag&
=\Mr(1)
-\sum_{i=1}^N \frac{1}{\psi'(s_i)}\frac{\gG(|s_i|+1)\gG(n)}{\gG(n+|s_i|+1)}
+\int_0^{1/2}(1-x)^{n-1}O(x^{|s_{N+1}|})\dd x 
\\\notag& \hskip16em{}
+\int_{1/2}^1O\bigpar{2^{-n}|r_N(x)|}\dd x 
\\\notag&
=\Mr(1)
-\sum_{i=1}^N \frac{\gG(|s_i|+1)}{\psi'(s_i)}\frac{\gG(n)}{\gG(n+|s_i|+1)}
+O(n^{-1-|s_{N+1}|})
.\end{align}
Using \eqref{av5} instead of \eqref{au7} in combination with
\eqref{cxED}, \eqref{lmr}, \eqref{au6}, and \eqref{aua3} yields 
  \begin{align}\label{av8}
  \Ex[D_n] = \frac{6}{\pi^2}h_{n-1} + \frac{\zeta(3)}{\zeta(2)^2}
-\sum_{i=1}^N \frac{\gG(|s_i|+1)}{\psi'(s_i)}\frac{\gG(n)}{\gG(n+|s_i|+1)}
+O(n^{-1-|s_{N+1}|}).
\end{align}
Since $N$ is arbitrary,
this shows  the asymptotic expansion \eqref{aw1}.

Finally, \eqref{aw2} follows from \eqref{aw1}.
In fact, for every fixed $b$, we have the
asymptotic expansion
\cite[5.11.13]{NIST}
\begin{align}\label{an2}
  \frac{\gG(n)}{\gG(n+b)}\sim \sum_{k=0}^\infty g_k(b) n^{-b-k},
\end{align}
for some coefficients $g_k(b)$ that can be calculated explicitly
\cite[5.11.15 and 17]{NIST}.
Substituting these expansions and \eqref{an1}
in \eqref{aw1} yields \eqref{aw2}.
\end{proof}

\begin{example}
  Recall from \refL{Lpsi} that
$s_1\doteq-0.567$ and $s_2\doteq-1.628$.
The first terms in \eqref{aw2} thus yield,
with $c_0$ given by \eqref{aua4} and $c_2=-1/(2\pi^2)$ by \eqref{an1},
\begin{align}\label{aw3}
  \Ex[D_n] =
\frac{6}{\pi^2}\log n
+c_0
-\frac{3}{\pi^2}n^{-1}
- \frac{\gG(|s_1|+1)}{\psi'(s_1)}n^{-|s_i|-1}
-\frac{1}{2\pi^2}n^{-2}
+O\bigpar{n^{-|s_1|-2}},
\end{align}
where $-|s_1|-1\doteq-1.567$.
The next terms are constants times
$n^{-|s_1|-2}$ and $n^{-|s_2|-1}$, with exponents $-2.567$ and $-2.628$.
Numerically, the coefficient of $n^{-|s_1|-1}$ is
\\
$-{\gG(|s_1|+1)}/{\psi'(s_1)}\doteq
-0.0943$.
\end{example}

\section{The expected hop-height $\E[L_n]$} \label{SEL1}

We next find a similar asymptotic expansion for the expectation of the 
hop-height  $L_n$ in discrete time,
extending the first terms given in \cite[Theorem 1.2]{beta1} by a different
method.

\begin{Theorem}\label{TL}
  We have the asymptotic expansion
  \begin{multline}\label{tl1}
  \Ex[L_n] \sim 
\frac{3}{\pi^2}h_{n-1}^2+\frac{\zeta(3)}{\zeta(2)^2}h_{n-1}
+\frac{\zeta(3)^2}{\zeta(2)^3}+\frac{1}{10}-\frac{3}{\pi^2}\psi'(n)
\\
+\sum_{i=1}^\infty \frac{\gG(|s_i|+1)}{(|s_i|+1)\psi'(s_i)}
\frac{\gG(n)}{\gG(n+|s_i|+1)}
\end{multline}
where $0>s_1>s_2>\dots$ are the negative roots of $\psi(s)=\psi(1)$.
Alternatively, we have
\begin{multline}\label{sw2x}
  \E[L_n]\sim
\frac{3}{\pi^2}\log^2n 
+ \Bigpar{\frac{\zeta(3)}{\zeta(2)^2}+\frac{\gam}{\zeta(2)}}\log n
+b_0
\\
+\sumk a_k n^{-k}\log n
+\sumk b_k n^{-k}
+ \sumj \sumk c_{j,k}n^{-|s_j|-k}
\end{multline}
for some computable constants $a_k$, $b_k$, $c_{j,k}$;
in particular,
\begin{align}\label{hw21}
  b_0 = 
  \frac{3\gam^2}{\pi^2}
+\frac{\zeta(3)}{\zeta(2)^2}\gam+
\frac{\zeta(3)^2}{\zeta(2)^3}
+\frac{1}{10}
\doteq
0.78234
.\end{align}
\end{Theorem}

\begin{Remark}\label{Rc0}
The coefficient for $\log n$ in
\eqref{sw2x} (found already in \cite{beta1}) 
equals the constant term $c_0$ in the asymptotic expansion \eqref{aw2} of
$\E[D_n]$. 
\end{Remark}

\begin{proof}
We write \eqref{cy3} as
\begin{align}\label{aud3}
\Ex[L_n] &
=\intoi H_{n}(x)\dd\gU(x)
,\end{align}
where we 
(substituting $j=k+1$)
define the function,
\begin{align}\label{aud4}
  H_n(x):= \sum_{k=1}^{n-1} h_{k}\binom {n-1}{k} x^{k}(1-x)^{n-1-k},
\qquad 0\le x\le 1.
\end{align}
To obtain a more tractable form of $H_n$ for our analysis, we note that
\begin{align}\label{aud5}
  h_k=\sum_{i=1}^k\frac{1}{i}=\sum_{i=1}^k\intoi u^{i-1}\dd u
=\intoi\frac{1-u^k}{1-u}\dd u.
\end{align}
Hence, 
\eqref{aud4} yields
(for convenience shifting the index to $n+1$)
\begin{align}\label{aud6}
  H_{n+1}(x)&
=\sum_{k=1}^{n}\intoi\ddx u\frac{1-u^k}{1-u}\binom {n}{k} x^{k}(1-x)^{n-k}
\\\notag&
=\intoi\frac{\ddx u}{1-u}
\sum_{k=0}^{n}\binom {n}{k} \bigpar{x^{k}-(ux)^k}(1-x)^{n-k}
\\\notag&
=\intoi\frac{\ddx u}{1-u}
\bigpar{1-(1-x+ux)^n}
\\\notag&
=\intoi
\bigpar{1-(1-xw)^n}
\frac{\ddx w}{w}
\\\notag&
=\intox
\bigpar{1-(1-y)^n}
\frac{\ddx y}{y}
.\end{align}
Since \eqref{aud4} yields $H_n(1)=h_{n-1}$
(which also follows by \eqref{aud6} and  \eqref{su3}), 
\eqref{aud6} yields also
\begin{align}
  \label{aud7}
H_{n}(x)
=H_{n}(1)-\int_x^1 \bigpar{1-(1-y)^{n-1}}\frac{\ddx y}{y}
=h_{n-1}-\int_x^1 \bigpar{1-(1-y)^{n-1}}\frac{\ddx y}{y}
.\end{align}

By \eqref{aud4}, $H_n(x)$ is a polynomial on $\oi$ with $H_n(0)=0$, and thus
the Mellin transform $\MHn(s)$ exists for $\Re s>-1$.
When $\Re s>0$, we have by \eqref{aud7} and \eqref{su1}
\begin{align}\label{so1}
  \MHn(s)&=\intoi x^{s-1}H_n(x)\dd x
\\\notag&
=\frac{h_{n-1}}{s}-\iint_{0<x<y<1}x^{s-1}\bigpar{1-(1-y)^{n-1}}y\qw \dd x\dd y
\\\notag&
=\frac{h_{n-1}}{s}-\frac{1}{s}\int_{0<y<1}y^{s-1}\bigpar{1-(1-y)^{n-1}} \dd y
\\\notag&
=\frac{1}{s}\bigpar{h_{n-1}-\Mfn(s)},
\end{align}
and this extends to $\Re s>-1$ by analytic continuation.
(The \rhs{} is analytic in this domain by \refL{L99} and \eqref{su3}.)
For $s=0$, the \rhs{} of
\eqref{so1} is interpreted as a limit in the standard way, giving
\begin{align}\label{so10}
\MHn(0)= -\Mfn'(0)
.\end{align}

We now use \eqref{aud3}, \eqref{aud6}, and the expansion \eqref{lmrN}
for the density $f(x)$ of $\gU$. 
First, by \eqref{aud3} and  \eqref{lmr},
\begin{align}\label{sv1}
  \E[L_n]
=
\frac{6}{\pi^2}\intoi H_n(x) \frac{1}{x}\dd x
+\intoi H_n(x) r(x)\dd x.
\end{align}
For the main term, we 
note that 
$\intoi H_n(x) x\qw\dd x=\MHn(0)$,
which by \eqref{so10} 
equals $-\Mfn'(0)$.
To compute this derivative, we use \eqref{su2} and make a Taylor expansion
of $\gG(s+1)\gG(n)/\gG(n+s)$ to obtain, recalling \eqref{no2} and
\eqref{no3}--\eqref{no4},
\begin{align}\label{sv4}
  \intoi H_n(x)\frac{1}{x}\dd x &
=\frac12 \frac{\ddx^2}{\dd s^2}\frac{\gG(s+1)\gG(n)}{\gG(n+s)}\Bigr|_{s=0}
\\\notag&
=\tfrac12\bigsqpar{\bigpar{\psi(1)-\psi(n)}^2+\psi'(1)-\psi'(n)}
\\\notag&
=\tfrac12 h_{n-1}^2 +\tfrac{\pi^2}{12}-\tfrac12\psi'(n)
.\end{align}

For the final term in \eqref{sv1}, we use  \eqref{aud7} and obtain,
recalling \eqref{aua3},
\begin{align}\label{sv5}
&  \intoi H_n(x) r(x)\dd x
\\\notag&
= h_{n-1}\intoi r(x)\dd x
-\iint_{0<x<y<1}\frac{1-(1-y)^{n-1}}y r(x)\dd x\dd y
\\\notag&
=\Mr(1) h_{n-1}
-\iint_{0<x<y<1}\frac{1}y r(x)\dd x\dd y
+\iint_{0<x<y<1}\frac{(1-y)^{n-1}}y r(x)\dd x\dd y
\\\notag&
=\Mr(1) h_{n-1}
+\intoi(\log x) r(x)\dd x
+\intoi\frac{(1-y)^{n-1}}y \int_0^yr(x)\dd x\dd y
.\end{align}
A differentiation under the integral sign in \eqref{mellin} shows that 
\begin{align}\label{r'0}
\intoi(\log x) r(x)\dd x=\Mrprime(1).  
\end{align}
For the final integral in \eqref{sv5}
we use the expansion \eqref{av3}.
Note that a term $x^s$ ($s\ge0$) in $r(x)$ when substituted into
this integral
yields
\begin{align}\label{sv6}
 \frac{1}{s+1}\intoi \frac{(1-y)^{n-1}}y y^{s+1}\dd y
=\frac{\gG(s+1)\gG(n)}{(s+1)\gG(n+s+1)}
.\end{align}
For the remainder term $r_N$, we split the integral into two parts as in
\eqref{av5}, and use \eqref{aurN} for $y\in(0,\frac12)$ and \eqref{aur2N} for
$y\in(\frac12,1)$. 
Hence, \eqref{sv5} and \eqref{av3} yield
\begin{align}\label{sv7}
&  \intoi H_n(x) r(x)\dd x
\\\notag&
=\Mr(1) h_{n-1}
+\Mrprime(1) 
+
\sum_{i=1}^N \frac{1}{\psi'(s_i)}
\frac{\gG(|s_i|+1)\gG(n)}{(|s_i|+1)\gG(n+|s_i|+1)}
+ O\bigpar{n^{-|s_{N+1}|-1}}.
\end{align}

To find $\Mrprime(1)$ we use \eqref{mr} and find by
a Taylor expansion of $\psi(s)$,
using $\psi'(1)=\zeta(2)=\pi^2/6$, $\psi''(1)=-2\zeta(3)$ and
$\psi'''(1)=6\zeta(4)=\pi^4/15$ 
(see \eqref{no3} and \eqref{psi'} or \cite[5.7.4]{NIST}),
\begin{align}\label{cxx2b}
  \Mrprime(1)
=
\frac{(\psi''(1))^2}{4(\psi'(1))^3}-\frac{\psi'''(1)}{6(\psi'(1))^2}
=\frac{\zeta(3)^2}{\zeta(2)^3}-\frac{\zeta(4)}{\zeta(2)^2}
=\frac{\zeta(3)^2}{\zeta(2)^3}-\frac{2}{5}
.\end{align}

We obtain \eqref{tl1} by \eqref{sv1}, \eqref{sv4}, \eqref{sv7},
\eqref{aua3}, and \eqref{cxx2b}.

Finally,
we obtain \eqref{sw2x} from \eqref{tl1} by substituting 
\eqref{an1}, the corresponding asymptotic expansion of $\psi'(n)$
%
%
(obtained from \eqref{5.11.2} by termwise differentiation),
and \eqref{an2}, and then rearranging the terms.
\end{proof}

\section{The length of $\CTCS(n)$}\label{SgL}

Recall from \eqref{qm1} that the length $\gL_n$ of $\CTCS(n)$ satisfies the identity
\begin{align}\label{qm12}
  \E[\gL_n]& 
= \intoi \frac{1}{x}\Bigpar{1-(1-x)^n-nx(1-x)^{n-1}}\dd\gU(x)
. \end{align}

We denote the integrand in the integral by
\begin{align}\label{qm2}
  \gl_n(x):=\bigpar{1-(1-x)^n-nx(1-x)^{n-1}}/x.
\end{align}
Then its Mellin transform is, by beta integrals and simple algebra,
\begin{align}\label{qm3}
  \M\gl_n(s)
&
=\intoi\Bigpar{x^{s-2}-x^{s-2}(1-x)^n-nx^{s-1}(1-x)^{n-1}}\dd x
\\\notag&
=\frac{1}{s-1}-\frac{\gG(s-1)\gG(n+1)}{\gG(n+s)}-n\frac{\gG(s)\gG(n)}{\gG(n+s)}
\\\notag&
=\frac{1}{s-1}\Bigpar{1-\frac{\gG(s+1)\gG(n+1)}{\gG(n+s)}},
\end{align}
first assuming $\Re s>1$ and then by analytic continuation for $\Re s>-1$
(the domain where $\M\gl_n(s)$ exist, for any $n\ge2$); note that $s=1$ is a
removable singularity in \eqref{qm3} and \emph{not} a pole.
In particular, \eqref{qm3} yields
\begin{align}\label{qm4}
    \M\gl_n(0)=-(1-n)=n-1.
\end{align}

We use \refL{LM} and obtain from \eqref{lmr},  \eqref{qm12}, and \eqref{qm2}
\begin{align}\label{qm5}
  \E[\gL_n] &=\intoi\gl_n(x)\gu(x)\dd x
=\frac{6}{\pi^2}\intoi\gl_n(x)x\qw\dd x+\intoi\gl_n(x)r(x)\dd x
.\end{align}
The main term here is, using \eqref{qm4},
\begin{align}\label{qm6}
\frac{6}{\pi^2}\intoi\gl_n(x)x\qw\dd x
=\frac{6}{\pi^2}\M\gl_n(0)
=\frac{6}{\pi^2}(n-1)
\end{align}
while the remainder term in \eqref{qm4} can be estimated as,
by arguments as in \eqref{au7},
\begin{align}\label{qm7}
&\intoi\gl_n(x)r(x)\dd x
\\\notag&\quad
=\intoi x\qw r(x)\dd x
-\intoi (1-x)^nx\qw r(x)\dd x
-n\intoi (1-x)^{n-1} r(x)\dd x
\\\notag&\quad
=\Mr(0)+O\bigpar{n^{-|s_1|}}
.\end{align}
Furthermore, $\psi(s)$ has a pole at $s=0$, and thus \eqref{mr} yields
\begin{align}\label{qm8}
  \Mr(0)= 0-\frac{6}{\pi^2}\cdot\frac{1}{-1}
= \frac{6}{\pi^2}.
\end{align}
Consequently, by \eqref{qm5}--\eqref{qm8},
\begin{align}\label{qm9}
\E[\gL_n] = \frac{6}{\pi^2}(n-1)+\frac{6}{\pi^2}+O\bigpar{n^{-|s_1|}}
=\frac{6}{\pi^2}n + O\bigpar{n^{-|s_1|}}.
\end{align}

We may easily extend the argument above and obtain the following full asymptotic
expansion.

\begin{theorem}\label{TgL}
\begin{align}\label{tgl}
  \E[\gL_n]
\sim
\frac{6}{\pi^2}n-
\sumi
\frac{\gG(|s_i|+2)}{|s_i|\psi'(s_i)}\frac{\gG(n+1)}{\gG(n+|s_i|+1)}.
\end{align}
Alternatively, for some coefficients $c_{i,k}$ that can be found explicitly,
  \begin{align}\label{tgl2}
  \Ex[\gL_n] \sim 
\frac{6}{\pi^2} n
+\sum_{j=1}^\infty\sum_{k=0}^\infty c_{j,k}\,n^{-|s_j|-k}
\end{align}
\end{theorem}
Note that the terms in the sum \eqref{tgl} have orders $n^{-|s_i|}$.

\begin{proof}
This is similar to previous proofs, so we omit some details. 

We may use \eqref{av3} in the last two integrals in \eqref{qm7};
then calculations similar to \eqref{av5} yield \eqref{tgl}.
Finally, \eqref{tgl} implies \eqref{tgl2} by \eqref{an2}.
\end{proof}

\section{Alternative proofs via a Parseval formula}\label{SParseval}

We now show that one may  skip the intermediate step of obtaining 
asymptotics for the measure $\gU$ by using the following version of
Parseval's formula (also called Plancherel's formula) 
for Mellin transforms: see \refApp{AParseval} for a proof and references.

Recall that an integral $\intoooo f(x)\dd x$ exists \emph{conditionally}
if $f$ is locally integrable and the symmetric limit
$\lim_{A\to\infty}\int_{-A}^A f(x)\dd x$ exists (and is finite); this limit
is then defined to be 
the integral $\intoooo f(x)\dd x$. 
We use the same terminology for line integrals $\intgs$.


\begin{Lemma}\label{LP}
Suppose that $f$ is a locally integrable function and $\mu$ a measure
on $\bbR_+$,
and that $\gs\in\bbR$ is such that $\intoo x^{\gs-1}|f(x)|<\infty$ 
and $\intoo x^{-\gs}\dd\mu<\infty$, i.e., the
Mellin transforms $\Mf(s)$ and $\Mmu(1-s)$ are defined  when 
$\Re s=\gs$. 
Suppose also  that the integral $\intgs\Mf(s)\M\mu(1-s)\dd s$ converges
at least conditionally.
Suppose further that $x^\gs f(x)$ is bounded and that $f$ is 
$\mu$-a.e.\ continuous.
Then 
\begin{align}\label{lp}
 \intoo f(x)\dd\mu(x) = \frac{1}{2\pi\ii}\intgs\Mf(s)\M\mu(1-s)\dd s.
\end{align}
\end{Lemma}

We may apply this lemma to \eqref{cxED} and similar formulas,
and then obtain an asymptotic expansion from the complex line integral in
\eqref{lp} by shifting the line of integration using  residue calculus.
This gives an alternative method to obtain the asymptotic expansions found
above. The new method is perhaps less intuitive that the method used above, 
and although it does not require \refL{LM}, it requires some technical
arguments and estimates in the complex plane, somewhat similar to the proof
of \refL{LM}.
On the other hand, granted these technical details, 
the method is straightforward and presents the resulting
asymptotics in a more structured form as a sum of residues
(see for example \eqref{qk4}).
This is perhaps the main advantage of the method;
it is not important for the simple
cases studied so far, but it will be essential in \refS{SEDk} when we 
look at the more complicated case of higher moments.

Parseval's formula for Mellin transforms has long been used to derive
asymptotic expansions for various integrals and integral transforms
in a way  similar to our use here, see for example
\cite{Handelsman-Lew} and \cite{Paris-Kaminski}, but this is to our
knowledge the first application to combinatorial probability.
(Mellin transforms are used in other ways in many combinatorial problems,
see for example \cite{FGD}.)

\subsection{Asymptotics of $\E [D_n]$: method 2} \label{SED2}

As a warmup, we return to $\E [D_n]$ and give a second proof of \refT{TD},
using \refL{LP} instead of \refL{LM}.
We begin with some simple estimates that will be used repeatedly.

\begin{lemma}\label{Lhw6}
For $n\ge2$ and any complex $s=\gs+\ii\tau$, we have
\begin{align}\label{hw6a}
\lrabs{  \frac{\gG(s)\gG(n)}{\gG(s+n)}}&
=\frac{\gG(n)}{|s(s+1)|\prod_{j=2}^{n-1}|s+j|}
\le \frac{\gG(n)}{|s(s+1)|\prod_{j=2}^{n-1}|\gs+j|}
\\\notag&
= \frac{\gG(\gs+2)\gG(n)}{|s(s+1)|\,\gG(\gs+n)}
.\end{align}
Hence:
\begin{romenumerate}
\item 
For a fixed $n\ge2$, uniformly for $\gs$ in any compact subset of $(-1,\infty)$,
\begin{align}\label{hw6b}
\lrabs{  \frac{\gG(s)\gG(n)}{\gG(s+n)}}&
=
O\bigpar{|s|^{-2}}
.\end{align}

\item 
For a fixed $\gs>-1$, uniformly for all $n$,
\begin{align}\label{hw6c}
\lrabs{  \frac{\gG(s)\gG(n)}{\gG(s+n)}}&
=O\bigpar{|s|^{-2}n^{-\gs}}
.\end{align}
\end{romenumerate}
\end{lemma}
\begin{proof}
  First, \eqref{hw6a} is elementary, using $\gG(z+1)=z\gG(z)$.
If $\gs>-1+\gd$ for some $\gd>0$, then $\gs\le C(\gs+1)\le C|s+1|$
for some $C=C(\gd)$,
and it follows that $|s|\le C|s+1|$ and thus $|s+1|\qw\le C |s|\qw$.
Hence \eqref{hw6b} follows from \eqref{hw6a}.
Furthermore, for a fixed $\gs$, $\gG(n)/\gG(n+\gs)\sim n^{-\gs}$ as \ntoo,
see \cite[5.11.12]{NIST} (or \eqref{an2}), and thus
$\gG(n)/\gG(n+\gs)= O\bigpar{n^{-\gs}}$. Hence, \eqref{hw6c} too follows from
\eqref{hw6a}.
\end{proof}

\begin{lemma}\label{Lhw7}
  \begin{romenumerate}
  \item \label{Lhw7A}
For $s=\gs+\ii\tau$ with any fixed real $\gs\notin\set{1-s_i:i\ge0}$, we have 
\begin{align}\label{hw7a}
  |\psi(1-s)-\psi(1)|\ge c 
\end{align}
for some $c=c(\gs)>0$.
\item\label{Lhw7B} 
The lower bound \eqref{hw7a} holds uniformly for $\gs$ in any compact subset
of $\bbR$ and $|\gt|\ge1$.
  \end{romenumerate}
\end{lemma}

\begin{proof}
The \lhs{} of \eqref{hw7a} is in both cases 
a continuous function of $s$ that is 
non-zero on the given sets of $s$ by \refL{Lpsi};
furthermore it tends to $\infty$ by \eqref{5.11.2}
as $|\tau|\to\infty$ (with $\gs$ fixed or in a compact set).
Hence \eqref{hw7a} holds on the given sets of $s$ by a compactness argument.
\end{proof}

\begin{proof}[Second proof of \refT{TD}]
We apply \refL{LP} to the measure $\gU$ and the function
$f_n$ in \eqref{fn}.
Note that $f_n$ is continuous
everywhere except at $1$, and that $\gU\set1=0$ since $\gU$ is
concentrated on $\oio$ (by \eqref{e7}, because each $P_{t,1}<1$ a.s.);
hence $f_n$ is continuous  $\gU$-a.e.\ as required. Note also that
$f_n(x)=O(x)$ on $(0,1)$, so $x^{\gs}f_n(x)$ is bounded for $\gs\ge-1$.
The Mellin transform $\Mf_n(s)$ 
is given by \refL{L99}; it is defined for
$\Re s>-1$.

Fix $n\ge2$ and 
let $s=\gs+\ii\tau$ for some fixed $\gs\in(-1,0)$.
Then \eqref{su2}  implies 
that 
\begin{align}\label{xed1}
|\Mfn(s)|
=\Bigabs{\frac{1}{s}+O\Bigpar{\frac{1}{|s|^2}}}
\sim \frac{1}{|s|}\sim \frac{1}{|\tau|}  
\end{align}
as $\tau\to\pm\infty$.
Furthermore,
\eqref{m1} and \eqref{5.11.2} show that, for $|\tau|\ge2$, say,
\begin{align}\label{xed2}
|\M{\gU}(1-s)|
=\bigabs{\log(1-s)+O(1)}\qw
=\bigpar{\log|\tau|+O(1)}\qw
.\end{align}
Consequently,
$|\Mf_n(s)\M{\gU}(1-s)|\sim1/(|\tau|\log|\tau|)$ as $\tau\to\pm\infty$, 
and thus $\Mf_n(s)\M{\gU}(1-s)$ 
is \emph{not} absolutely integrable on any line $\Re s=\gs$.
Nevertheless, 
 $\Mf_n(s)\M{\gU}(1-s)$ is conditionally integrable on the line $\Re s=\gs$.
This can easily be shown, but we postpone this since it will follow
from the calculations below.
(Conditional integrability is all that we need, but the absence of absolute
integrability is a nuisance that complicates the arguments.)

Assuming this conditional integrability, we have verified
the conditions of Lemma \ref{LP}, 
and consequently \eqref{lp}
holds, which by \eqref{cxED}, \eqref{m1}, and \eqref{su2} yields,
for $-1<\gs<0$,
\begin{align}\label{qk1}
  &\E[D_n]
=\frac{1}{2\pi\ii}\intgs\Mfn(s)\M{\gU}(1-s)\dd s
=\frac{1}{2\pi\ii}\intgs\frac{\Mfn(s)}{\psi(1-s)-\psi(1)}\dd s
\\&\label{qk2}
=\frac{1}{2\pi\ii}\intgs\frac{1}{s(\psi(1-s)-\psi(1))} \dd s
-\frac{1}{2\pi\ii}\intgs\frac{\gG(s)\gG(n)/\gG(n+s)}{\psi(1-s)-\psi(1)}
\dd s,
\end{align}
where the second integral in \eqref{qk2} (but not the first)
is absolutely convergent, by \eqref{hw6b} and \eqref{hw7a}.
To treat the the first integral in \eqref{qk2} we split it into three as
\begin{multline}\label{hw1}
\intgs\frac{1}{s}\Bigpar{\frac{1}{\psi(1-s)-\psi(1)}-\frac{1}{\log(1-s)}}\dd s
\\
+ \intgs\Bigpar{\frac{1}{s}+\frac{1}{1-s}}\frac{1}{\log(1-s)}\dd s
- \intgs\frac{1}{1-s}\frac{1}{\log(1-s)}\dd s.
\end{multline}
The third integral in \eqref{hw1} has the primitive function $-\log\log(1-s)$.
For $s=\gs+\ii\gt$ with $\gs<0$, we have $|1-s|\ge1+\gs>1$ and thus
$\log(1-s)$ lies in the right half-plane.
Furthermore, as $\tau\to\pm\infty$,
$  \log(1-s) = \log|\tau|+O(1)$
and thus $\log\log(1-s)=\log\log|\tau|+o(1)$.
It follows that the third integral in \eqref{hw1} is conditionally
integrable, and that its value is 0.
(Note that it is important here that the definition uses 
the limit of symmetric integrals $\int_{-A}^A$.)

We will show that the first two integrals in \eqref{hw1} converge
absolutely. This means that all terms in \eqref{hw1} are well-defined,
and that we can sum them to obtain the first integral in \eqref{qk2},
which in turn means that this integral and the integrals in \eqref{qk1} are
conditionally convergent, which justifies the use of \refL{LP} above.

For the first integral in \eqref{hw1}, we note first that 
the integrand is analytic
in the domain $\Re s<0$
since $\psi(1-s)-\psi(1)\neq0$ there by \refL{Lpsi},
and that in this domain we have
by \eqref{5.11.2}
$\psi(1-s)=\log(1-s)+O(1)$
and thus, uniformly in any half space $\Re s\le\gd<0$,
\begin{align}
  \label{hw2}
&\frac{1}{\psi(1-s)-\psi(1)}-\frac{1}{\log(1-s)}
=\frac{O(1)}{(\psi(1-s)-\psi(1))\log(1-s)}
\\\notag&\qquad
=O\Bigpar{\frac{1}{|\log(1-s)|^2}}
=O\Bigpar{\frac{1}{\log^2|1-s|}}
=O\Bigpar{\frac{1}{\log^2(2+\tau)}}
.\end{align}
It follows that the integrand in the first integral can be bounded for 
$\Re s\le-\gd$ by
\begin{align}\label{hw3}
\lrabs{\frac{1}{s}\Bigpar{\frac{1}{\psi(1-s)-\psi(1)}-\frac{1}{\log(1-s)}}}
\le
\frac{C(\gd)}{(|\gs|+|\tau|)\log^2(2+\tau)}
.\end{align}
Hence the integral converges absolutely for every $\gs<0$.
Furthermore, \eqref{hw3} implies also that we may shift the line of integration,
and thus the value of the integral is the same for all $\gs<0$;
moreover, \eqref{hw3} and dominated convergence shows that as
$\gs\to-\infty$, the value converges to 0.
Consequently, the first integral in \eqref{hw1} vanishes for every $\gs<0$.

For the second integral in \eqref{hw1} we argue similarly, but simpler.
 In any half space $\Re s\le\gd<0$, we have 
$|\log(1-s)|\ge\log|1-s|\ge\log(1+\gd)$ and thus 
\begin{align}
  \label{hw4}
\Bigpar{\frac{1}{s}+\frac{1}{1-s}}\frac{1}{\log(1-s)}
&=
\frac{1}{s(1-s)}\frac{1}{\log(1-s)}
=O\Bigpar{\frac{1}{|s(1-s)|}}
=O\Bigpar{\frac{1}{|s|^2}}
.\end{align}
It follows again that the integral is absolutely convergent for every
$\gs<0$ and that we may let $\gs\to-\infty$ and conclude that the integral
vanishes. 

This completes the proof that the first integral in \eqref{qk2} is
conditionally convergent, and shows also that it is 0.
Hence, \eqref{qk1}--\eqref{qk2} finally simplifies to
\begin{align}\label{hw5}
  &\E[D_n]
=-\frac{1}{2\pi\ii}\intgs\frac{\gG(s)\gG(n)/\gG(n+s)}{\psi(1-s)-\psi(1)}
\dd s,
\end{align}
valid for any $n\ge2$ and $\gs\in(-1,0)$.

Denote the integrand in \eqref{hw5} by $g(s)$; then
$g(s)$  is analytic in \set{\Re s>-1} except for poles
$\set{1-s_i: i\ge0}$ (including $1-s_0=0$).
For any $n\ge2$ and 
$\gs\notin\set{1-s_i:i\ge0}$ with $\gs>-1$,
\eqref{hw6b} and \eqref{hw7a} show that the 
integral in \eqref{hw5} converges absolutely.
Moreover, if we assume $|\gt|\ge1$, then \eqref{hw6b} and \refL{Lhw7}\ref{Lhw7B}
show that the integrand in \eqref{hw5} is $O\bigpar{|s|\qww}$ uniformly for
$\gs$ in any compact subset of $(-1,\infty)$,
and it follows that we may move the line of integration from a
$\gs\in(-1,0)$ to a (large) positive $\gs$,
picking
up $-2\pi\ii$ times the residues at the passed poles.

So far we have argued with a fixed $n$; now we let $n\ge2$ be arbitrary.
Then, for a fixed $\gs=\Re s>0$ not in the set \set{1-s_i:i\ge0} of poles, 
\eqref{hw6c} and \eqref{hw7a} imply
that the  integral in \eqref{hw5} is $O\bigpar{n^{-\gs}}$.
Consequently, \eqref{hw5} implies that $\E[D_n]$ has an asymptotic
expansion consisting of the residues of $g(s)$ at its poles:
\begin{align}\label{qk4}
  \E[D_n]
=
\sum_{i=0}^N \Res_{s=1-s_i} \frac{\gG(s)\gG(n)/\gG(n+s)}{\psi(1-s)-\psi(1)}
+ O\bigpar{n^{-\gs}}
\end{align}
for any fixed $N\ge0$ and  $\gs<1-s_{N+1}$.
The first pole $0=1-s_0$ has order $2$ 
(since also $\gG(s)$ has a pole there)
and a straightforward calculation, using \eqref{no2}--\eqref{no4} and 
$\psi''(1)=-2\zeta(3)$ (which follows from \eqref{psi'} or \cite[5.7.4]{NIST})
yields
\begin{align}
\Res_{s=0}\frac{\gG(s)\gG(n)/\gG(n+s)}{\psi(1-s)-\psi(1)}
  =\frac{\psi(n)-\psi(1)}{\psi'(1)}-\frac{\psi''(1)}{2\psi'(1)^2}
=\frac{h_{n-1}}{\zeta(2)}+\frac{\zeta(3)}{\zeta(2)^2}.
\end{align}

For any other pole $1-s_i=1+|s_i|$ ($i\ge1$), $g(s)$ has a pole of order
$1$, with residue
\begin{align}\label{qk5}
\Res_{s=1-s_i}\frac{\gG(s)\gG(n)/\gG(n+s)}{\psi(1-s)-\psi(1)}
=-\frac{\gG(|s_i|+1)\gG(n)/\gG(n+|s_i|+1)}{\psi'(s_i)}
.\end{align}
Hence, \eqref{qk4} yields \eqref{aw1}.
Finally, \eqref{aw2} follows from \eqref{aw1} as before,
which completes the proof of \refT{TD}.
\end{proof}

\subsection{The hop-heights $L_n$: method 2}\label{SEL2}
We find it instructive to also use \refL{LP} to give another proof of
\refT{TL}. 
\begin{proof}[Second proof of \refT{TL}]
By \eqref{aud4}, $H_n(x)$ is a polynomial on $\oi$ with $H_n(0)=0$, and
the Mellin transform $\MHn(s)$ exists for $\Re s>-1$;
also, $x^\gs H_n(x)$ is bounded for $\gs>-1$.
Hence, \refL{LP} applies to $H_n(x)$ and $\gU$ and any $\gs\in(-1,0)$,
provided the second integral in \eqref{lp} exists at least conditionally.

We thus obtain, by \eqref{aud3}, \refL{LP}, \eqref{so1}, and
\eqref{m1}, for $-1<\gs<0$, 
\begin{align}\label{so2}
&  \E[L_n]=
\frac{1}{2\pi\ii}\intgs \MHn(s)\M{\gU}(1-s)\dd s
=\frac{1}{2\pi\ii}\intgs \frac{h_{n-1}-\Mfn(s)}{s\xpar{\psi(1-s)-\psi(1)}}\dd s
\\&\label{so2b}
=\frac{h_{n-1}}{2\pi\ii}\intgs \frac{1}{s\xpar{\psi(1-s)-\psi(1)}}\dd s
-\frac{1}{2\pi\ii}\intgs \frac{\Mfn(s)}{s\xpar{\psi(1-s)-\psi(1)}}\dd s
\end{align}
provided the last two integrals exist at least conditionally.
In fact, the first integral in \eqref{so2b} is the same as in \eqref{qk2}, 
which we have shown converges and equals 0.
Furthermore, \eqref{xed1} and \eqref{hw7a} show that for any fixed
$\gs\in(-1,0)$,
the integrand in the second integral in \eqref{so2b} is $O\bigpar{|s|\qww}$
and thus this integral is absolutely convergent.
Consequently, \eqref{so2}--\eqref{so2b} are justified, and simplify to,
using \eqref{su2}, and with absolutely convergent integrals,
\begin{align}\label{so3}
  \E[L_n]&=
-\frac{1}{2\pi\ii}\intgs \frac{\Mfn(s)}{s\xpar{\psi(1-s)-\psi(1)}}\dd s
\\\notag&
=-\frac{1}{2\pi\ii}\intgs \frac{1}{s^2\xpar{\psi(1-s)-\psi(1)}}\dd s
+\frac{1}{2\pi\ii}\intgs 
 \frac{\gG(s)\gG(n)/\gG(n+s)}{s\xpar{\psi(1-s)-\psi(1)}}\dd s
.\end{align}
In the first integral on the \rhs, we may again shift $\gs\to-\infty$;
the integrand is analytic for $\Re s<0$, 
and $O\bigpar{|s|\qww}$ is any halfplane $\Re s\le\gd<0$,
and it follows (by dominated convergence) that the limit, 
and thus the integral, is $0$.
Consequently, \eqref{so3} simplifies to
\begin{align}\label{so4}
  \E[L_n]&
=
\frac{1}{2\pi\ii}\intgs 
 \frac{\gG(s)\gG(n)/\gG(n+s)}{s\xpar{\psi(1-s)-\psi(1)}}\dd s
,\qquad -1<\gs<0
.\end{align}
In the half-plane $\Re s>-1$, the integrand in \eqref{so4} has poles at 
0 and $1-s_i=1+|s_i|$ for $i\ge1$; there is a triple pole at 0, and all
other poles are simple. As in \refS{SED2}, we may shift $\gs$ to a large
positive value; in fact, the integrand in \eqref{so4} is $g(s)/s$,
with $g(s)$ as in \refSS{SED2},
so we may just reuse the estimates in \refS{SED2}. 
This yields
\begin{align}\label{ql4}
  \E[L_n]
=
-\sum_{i=0}^N \Res_{s=1-s_i} \frac{\gG(s)\gG(n)/\gG(n+s)}{s(\psi(1-s)-\psi(1))}
+ O\bigpar{n^{-\gs}}
\end{align}
for any fixed $N\ge0$ and  $\gs<1-s_{N+1}$.
  The residue at 0 is by straightforward (but tedious) calculation 
$-1$ times
\begin{align}\label{hw22}
\frac{3}{\pi^2}h_{n-1}^2+\frac{\zeta(3)}{\zeta(2)^2}h_{n-1}
+\frac{\zeta(3)^2}{\zeta(2)^3}+\frac{1}{10}-\frac{3}{\pi^2}\psi'(n)
\end{align}
and the residues at the simple poles are immediate.
Hence, \eqref{ql4} yields \eqref{tl1}.
\end{proof}

Note that in the proofs of \refTs{TD} and \ref{TL}
just given, the significant differences between 
$\E [D_n]$ and $\E[ L_n]$ is that in the proof of \refT{TD} we have a pole at 0
of order 2, while the corresponding function in the proof of \refT{TL} has a
pole of order 3. This explains the different powers of $\log n$ in the
leading term; it also explains why the calculations for $\E[ L_n]$ in the
first proof in \refS{SEL1} are somewhat more complicated that for $\E[ D_n]$ in
\refS{SED1}.

\subsection{The length $\gL_n$: method 2}\label{SgL2}
\refL{LP} yields also an alternative proof of 
\refT{TgL}.  

\begin{proof}[Second proof of \refT{TgL}]
This is similar to the proofs above, so we omit some details.
Recalling \eqref{qm12}--\eqref{qm3},
we may apply \refL{LP} to $\gl_n$ and $\gU$, 
again for any $n\ge2$ and $-1<\gs<0$.
As in \refS{SED2}, this leads to 
(after discarding one conditionally convergent integral that is 0),
for $-1<\gs<0$,
\begin{align}\label{qp2}
  \E\gL_n&
=-\frac{1}{2\pi\ii}
\intgs \frac{\gG(s+1)\gG(n+1)/\gG(n+s)}{(s-1)(\psi(1-s)-\psi(1)) }\dd s.
\end{align}
We again shift the
line of integration to a large $\gs$,
noting that in the half-plane $\Re s>-1$, the integrand has 
poles at $\set{1-s_i:i\ge0}$, while
$s=1$ is a removable singularity and not a pole since $\psi(1-s)$
has a pole there.
Hence we obtain, similarly to \eqref{qk4}, 
\begin{align}\label{qp4}
  \E[\gL_n]
=
\sum_{i=0}^N \Res_{s=1-s_i} 
\frac{\gG(s+1)\gG(n+1)/\gG(n+s)}{(s-1)(\psi(1-s)-\psi(1))}
+ O\bigpar{n^{1-\gs}}
\end{align}
for any fixed $N\ge0$ and  $\gs<1-s_{N+1}$.
All poles are simple, so it is straightforward to compute the residues.
In particular, the residue at 0 is $n/\psi'(1)=n/\zeta(2)$, and we 
obtain \eqref{tgl}. 
\end{proof}

\section{Higher moments of $D_n$}\label{SEDk}

The results on $\E[D_n]$ above may be extended to higher moments.
For any integer $k\ge1$, we have by \eqref{cx3}
\begin{align}\label{az1}
  \E[D_n^k]&
= k\intoo t^{k-1} \Pr(D_n>t)
= k\intoo t^{k-1} \E\bigsqpar{1-(1-P_{t,1})^{n-1}}\dd t.
\end{align}
Thus, if we define an infinite measure on $\oio$ by
\begin{align}\label{az2}
  \gU_k:=k\intoo t^{k-1}\cL(P_{t,1}) \dd t
\end{align}
(so $\gU_1=\gU$ defined in \eqref{e7}), then
\begin{align}\label{az3}
  \E[D_n^k]&
=\intoi\bigsqpar{1-(1-x)^{n-1}}\dd \gU_k(x).
\end{align}

As in the special case $k=1$ in \eqref{m1}, we obtain the Mellin transform
(i.e., moments) of $\gU_k$ from \eqref{b5}:
\begin{align}\label{az4}
\intoi x^{s-1}\dd\gU_k(x)&=
k\intoi t^{k-1}\intoo x^{s-1}\dd\cL(P_{t,1})(s)\dd t
= k\intoi t^{k-1}\E [P_{t,1}^{s-1}]\dd t
\\\notag&
= k\intoi t^{k-1}e^{-t(\psi(s)-\psi(1))}\dd t
\\\notag&
=k!\, \bigpar{\psi(s)-\psi(1)}^{-k},
\qquad \Re s>1.
\end{align}

Again, the Mellin transform extends to a meromorphic function in the complex
plane, with poles at $1 > s_1 > s_2 > \dots$ given by \refL{Lpsi}.

\subsection{Asymptotics of $\E [D_n^k]$}\label{SSEDk}
Fix $k\ge2$.
We apply \refL{LP} to $\gU_k$ and the function 
$f_n(x)$ in \eqref{fn} as for the case $k=1$ in \refS{SED2}.
If $s=\gs+\ii\tau$ with a fixed 
$\gs\in(-1,0)$, then \eqref{az4} and \eqref{xed2} show that
\begin{align}\label{xed2k}
|\MgU_k(1-s)|
= k!\, |\MgU(1-s)|^k
=\bigpar{\log|\tau|+O(1)}^{-k},
.\end{align}
which combined with \eqref{xed1} shows that 
$\Mf_n(s)\MgU_k(1-s)$ is absolutely integrable on the line $\Re s=\gs$.
(This simplifies the argument  needed for $k=1$ in \refS{SED2}.)
The other conditions  of \refL{LP} are satisfied as in \refSS{SED2}, 
in particular, $f_n$ is continuous  $\gU_k$-a.e.\ since $\gU_k\set1=0$.
 Consequently \eqref{lp}
holds, which by \eqref{az3}, \eqref{fn}, and  \eqref{az4} yields 
\begin{align}\label{kb1}
  \E[D_n^k]&
=\frac{1}{2\pi\ii}\intgs\Mfn(s)\M{\gU}_k(1-s)\dd s
\\\notag&
=\frac{k!}{2\pi\ii}\intgs\frac{\Mfn(s)}{(\psi(1-s)-\psi(1))^k}\dd s,
\qquad -1<\gs<0.
\end{align}
We split $\Mfn(s)$ into two parts according to \eqref{su2}; thus
\eqref{kb1} yields
\begin{align}\label{kb2}
  \E[D_n^k]=&
\frac{k!}{2\pi\ii}\intgs\frac{1}{s(\psi(1-s)-\psi(1))^k}
\dd s
\\\notag&\qquad{}
-\frac{k!}{2\pi\ii}\intgs\frac{\gG(s)\gG(n)/\gG(n+s)}{(\psi(1-s)-\psi(1))^k}
\dd s,
\qquad -1<\gs<0.
\end{align}
In the first integral in \eqref{kb2}, the integrand 
is
analytic in the half-plane \set{\Re s<0}. Furthermore, in this half-plane 
$\psi(1-s)=\log(1-s)+O(1)=\log(|s|+1)+O(1)$, see again \eqref{5.11.2}.
It follows first that we may move the line of integration to any $\gs<0$,
and then that as $\gs\to-\infty$, the integral tends to 0 by dominated
convergence. 
(Again, the case $k\ge2$ is simpler that $k=1$ here, although the conclusion
is the same.)
Consequently, the first integral in \eqref{kb2} is 0, and thus
we have, as for the case $k=1$ in \eqref{hw5},
\begin{align}\label{piglet}
  \E[D_n^k]=&
-\frac{k!}{2\pi\ii}\intgs\frac{\gG(s)\gG(n)/\gG(n+s)}{(\psi(1-s)-\psi(1))^k}
\dd s,
\end{align}
for any $n\ge2$ and $\gs\in(-1,0)$.

We then argue as in \refSS{SED2}, and
move the line of integration in \eqref{piglet} to a (large) positive $\gs$;
this is again justified by the estimates in \refLs{Lhw6} and
\ref{Lhw7}. 
The integrand, $g_k(s)$ say, is again analytic in \set{\Re s>-1} except for
poles at 0 and $\set{1-s_i=|s_i|+1: i\ge1}$,
and we obtain again
an asymptotic
expansion consisting of the residues of $g_k(s)$ at its poles:
\begin{align}\label{kb4}
  \E[D_n^k]
=
k!\sum_{i=0}^N \Res_{s=1-s_i} 
\frac{\gG(s)\gG(n)/\gG(n+s)}{(\psi(1-s)-\psi(1))^k}
+ O\bigpar{n^{-\gs}}
\end{align}
for any fixed $N\ge0$ and  $\gs<1-s_{N+1}$.
The first pole $0=1-s_0$ has order $k+1$ 
so the residue there can be
obtained by computing the coefficients of the 
singular part of the Laurent expansion of
$\gG(s)/\bigpar{(\psi(1-s)-\psi(1))^k}$  
and the $k+1$ first Taylor coefficients of $\gG(n)/\gG(n+s)$ at $s=0$;
it follows from \eqref{no2} that
the latter coefficients are by polynomials in $\psi(n)$ and the derivatives
$\psi'(n),\dots,\psi^{(k-1)}(n)$; these may 
by \eqref{5.11.2}
all be expanded further into
asymptotic expansions involving $\log n$ (for $\psi(n)$ only) and powers of
$n\qw$.
The leading term will be a constant times $\log^kn$, in fact, it is
easily seen to be 
(including the factor $k!$ in \eqref{kb4})
\begin{align}\label{kc1}
\psi'(1)^{-k}\log^k n=\bigpar{\tfrac{6}{\pi^2}\log n}^k.  
\end{align}
For any other pole $1-s_i=1+|s_i|$ ($i\ge1$), $g_k(s)$ has a pole of order
$k$,
and the contribution from this pole to \eqref{kb4} is given by some
combination (with computable coefficients) of the $k$ first Taylor
coefficients of $\gG(n)/\gG(n+s)$ at $s=1+|s_i|$.
The dominant term will be 
a (nonzero) constant times 
$\psi(n)^{k-1}\gG(n)/\gG(n+1+|s_i|)\sim(\log n)^{k-1}n^{-|s_i|-1}$,
so the entire residue is of this order.
It follows that the expansion \eqref{kb4} is an asymptotic expansion of the
type defined in \eqref{asy1}--\eqref{asy2} with terms of successively
smaller order; hence we may write \eqref{kb4} as
\begin{align}\label{kb5}
  \E[D_n^k]
\sim
k!\sum_{i=0}^\infty \Res_{s=1-s_i} 
\frac{\gG(s)\gG(n)/\gG(n+s)}{(\psi(1-s)-\psi(1))^k}
,\end{align}
where the residues may calculated and expanded as above, leading to an
asymptotic expansion containing terms that are constants times
$(\log n)^k$ (the leading term),
$(\log n)^j n^{-\ell}$ ($0\le j\le k-1$, $\ell\ge0$),
and $(\log n)^j n^{-|s_i|-\ell}$ ($i\ge1$, $0\le j\le k-1$, $\ell\ge1$);
the coefficients may be computed as sketched above.
In particular, we obtain by collecting terms:
\begin{theorem}\label{TDk}
For any fixed integer $k\ge1$, 
we have as \ntoo{} the asymptotic expansion \eqref{kb5}, and in particular
  \begin{align}\label{kb6}
  \E[D_n^k] = p_k(\log n)+ O\Bigpar{\frac{\log^{k-1}n}{n}}
\end{align}
for some (computable) polynomial $p_k$ of degree $k$, with leading term
\eqref{kc1}.
\end{theorem}

For example,
  for $k=2$, the residue at 0 in \eqref{kb5} is, including the factor $2!$
and recalling \eqref{no4}
\begin{align}\label{kb7}
&  \frac{h_{n-1}^2}{\zeta(2)^2}+\frac{4\zeta(3)}{\zeta(2)^3}h_{n-1}
+\frac{6\zeta(3)^2}{\zeta(2)^4}-\frac{18}{5\pi^2}-\frac{1}{\zeta(2)^2}\psi'(n)
.\end{align}
Hence, \eqref{kb5} yields, using \eqref{au79},
\begin{align}\label{kb8}
\E[D_n^2]
&
=\frac{1}{\zeta(2)^2}h_{n-1}^2
+ \frac{4\zeta(3)}{\zeta(2)^3}h_{n-1}
+\frac{6\zeta(3)^2}{\zeta(2)^4}
-\frac{18}{5\pi^2}
+ O\Bigpar{\frac{1}{n}}
\\\notag&
=\frac{1}{\zeta(2)^2}\log^2n
+\Bigpar{\frac{2\gamma}{\zeta(2)^2} 
+\frac{4\zeta(3)}{\zeta(2)^3}}\log n
\\\notag&\qquad
+\frac{\gamma^2}{\zeta(2)^2}
+ \frac{4\zeta(3)}{\zeta(2)^3}\gamma
+\frac{6\zeta(3)^2}{\zeta(2)^4}
-\frac{18}{5\pi^2}
+ O\Bigpar{\frac{\log n}{n}}
.\end{align}

Combining \eqref{kb8} and \eqref{aw1}, we obtain the following result;
the leading term was found by a different method  in \cite[Theorem 1.1]{beta1}.
\begin{theorem}\label{TvarD}
The variance of $D_n$ is, as \ntoo,
  \begin{align}\label{kb9}
\var[D_n]
&
=
 \frac{2\zeta(3)}{\zeta(2)^3}h_{n-1}
+\frac{5\zeta(3)^2}{\zeta(2)^4}
-\frac{18}{5\pi^2}
+ O\Bigpar{\frac{\log n}{n}}
\\\notag&
=
\frac{2\zeta(3)}{\zeta(2)^3}\log n
+ \frac{2\zeta(3)}{\zeta(2)^3}\gamma
+\frac{5\zeta(3)^2}{\zeta(2)^4}
-\frac{18}{5\pi^2}
+ O\Bigpar{\frac{\log n}{n}}
.\end{align}
\end{theorem}

\subsection{Asymptotics of $\gU_k$}\label{SSgUk}
The argument above did not use any properties of $\gU_k$ except its Mellin
transform and $\gU_k\set1=0$. 
For completeness, we also consider results analoguous to \refL{LM}.

We may invert the Mellin transform by the method in \refS{SgU}
and show that $\gU_k$ is absolutely continuous
and obtain an asymptotic expansion of its density $\nu_k(x)$ as $x\downto0$.
However, now each pole has order $k$, which makes the Mellin inversion a bit
more complicated. 

Consider for simplicity first the case $k=2$. Then \eqref{az4} shows that
$\gU_2$ has the Mellin transform $2\xpar{\psi(s)-\psi(1)}^{-2}$,
which has  double poles at $1$ and every $s_i$. 
We argue as in \refS{SgU} (to which we refer for omitted details)
and consider the finite measure $\nu$ on
$(0,1)$
defined by $\ddx\nu(x)=x\dd\gU_2(x)$; this shifts the Mellin transform
to 
\begin{align}\label{ka1}
\widetilde{\nu}(s)=\widetilde{\gU_2}(s+1)=\frac{2}{(\psi(s+1)-\psi(1))^2},
\end{align}
which has a double pole at 0 with singular part 
\begin{align}\label{ka2}
\frac{2}{\zeta(2)^{2}}s^{-2}+\frac{4\zeta(3)}{\zeta(2)^{3}} s^{-1}.
\end{align}
We therefore define 
$\nu_0$ as the measure on $\oio$ with density
\begin{align}\label{ka3}
h_0(x):= \frac{2}{\zeta(2)^2}(-\log x) + \frac{4\zeta(3)}{\zeta(2)^3}
.\end{align}
Then $\nu_0$ has Mellin transform exactly  \eqref{ka2} (for $\Re s>0)$, 
and thus the signed measure
$\nux:=\nu-\nu_0$ has a Mellin transform which extends to a meromorphic
function in $\bbC$ with  poles only at
$\set{s_i-1:i\ge1}$ (all double).
Again, the Mellin transform is not integrable on any vertical lines, 
so we use again 
the trick of considering $\tnux'(s)$, which is 
the Mellin transform of $\log(x)\dd\nux(x)$, and which is
integrable on any vertical line not containing a pole $s_i$. 
Thus the Mellin inversion formula \eqref{m8} applies.
By, as in \refS{SgU},  shifting the line of integration and undoing the
modifications $\gU_2\to\nu\to\nux\to(\log x)\dd\nux$ 
used in the argument,
it follows that $\gU_2$ is absolutely continuous with a
density 
\begin{align}\label{ka4}
\gu_2(x)&
:=\frac{\ddx\gU_2}{\ddx x}
=\frac{1}{x}h_0(x)+\frac{1}{x\log x}O\bigpar{x^{-s_1+1-\eps}}
\\\notag&\phantom:
= \frac{2}{\zeta(2)^2}\cdot\frac{-\log x}{x} 
+ \frac{4\zeta(3)}{\zeta(2)^3}\cdot\frac{1}{x}
+O\bigpar{x^{|s_1|-\eps}|\log x|\qw},
\qquad 0<x<1,
\end{align}
for any $\eps>0$.
We may improve this to a full asymptotic expansion as in 
\eqref{lm++}
by caculating contributions from further poles;
note that we now get for each pole $s_i$ two terms,  with
$x^{-s_i}$ and $(\log x)x^{-s_i}$.

We now may use this and \eqref{az3} to obtain the asymptotical expansion
\eqref{kb8} for $\E[D_n^2]$, but the calculations are  more
complicated than in \refS{SED1} since we now have more terms, and we leave
the details to the interested reader.

From the main term in \eqref{ka4} we obtain, 
using both \eqref{su1} and the result of differentiating it,
\begin{align}\label{ka5}
&  \intoi\bigsqpar{1-(1-x)^{n-1}}
\Bigpar{\frac{2}{\zeta(2)^2}\cdot\frac{-\log x}{x} 
+ \frac{4\zeta(3)}{\zeta(2)^3}\cdot\frac{1}{x}}
\dd x
\\\notag&\qquad
=-\frac{2}{\zeta(2)^2}\Mfn'(0)
+ \frac{4\zeta(3)}{\zeta(2)^3}\Mfn(0).
\end{align}
For the remainder term in \eqref{ka4} we have essentially
the same estimate as in \eqref{au7} (with a different constant).

We have computed $\Mfn(0)=h_{n-1}$ in \eqref{su4},
and $-\Mfn'(0)$ in \eqref{sv4}.
Since $\psi'(n)=O(1/n)$, as a consequence of \eqref{5.11.2},
we find from \eqref{az3}
and \eqref{ka5}
\begin{align}\label{ka7}
\E[D_n^2]
&
=\frac{1}{\zeta(2)^2}h_{n-1}^2
+ \frac{4\zeta(3)}{\zeta(2)^3}h_{n-1}
+ O(1),
\end{align}
as already shown (with more precision) in \eqref{kb8}.
We may also obtain further terms, and a full asymptotic expansion,
but the method of \refSS{SSEDk} seems preferable.

We may treat moments of any order $k\ge2$ in the same way,
in principle with full asymptotic expansions. 
In general, the Mellin transform \eqref{az4} has
poles of order $k$, and the main term 
(for our purposes at least, i.e., for small $x$)
of the density $\gu_k(x)$ of $\gU_k$ will be of the form
$q_k(\log x)/x$, where $q_k$ is a polynomial of degree $k-1$;
the leading term of $\gu_k(x)$ is, more precisely,
$k\zeta(2)^{-k}(-\log x)^{k-1}/x$.
Substituting this in \eqref{az3} and
using \eqref{su1}--\eqref{su2} as above
yields, 
unsurprisingly, the main term
\begin{align}\label{ka9}
\E[D_n^k]\sim 
k\zeta(2)^{-k}\Bigpar{-\frac{\ddx}{\ddx s}}^{k-1}\Mfn(s)\Bigr|_{s=0}
\sim \Bigpar{\frac{h_{n-1}}{\zeta(2)}}^k
\sim \Bigpar{\frac{6}{\pi^2}\log n}^k.
\end{align}
Further terms can be obtained in a
straightforward way, but again the method in \refSS{SSEDk} seems preferable.

\section{Moment generating function of $D_n$}\label{Smgf}
Finally, we consider the moment generating function $\E[e^{zD_n}]$ of
$D_n$. Note first that conditioned on $\DTCS(n)$, i.e., 
on the structure of the tree but forgetting the edge lengths, 
the height $D_n$ is a sum of a finite number of exponential random
variables, each with expectation $1/h_{m-1}\le1$ for some $m\ge2$;
it follows that the moment generating function $\E[e^{zD_n}]$ 
exists for $\Re z<1$,
and thus is an analytic function there.

\begin{remark}\label{Rmgf1}
The same argument shows that $\E [e^{D_n}]=\infty$ for every $n\ge2$, 
since with positive probability leaf 1 belongs to a clade of size $m=2$; 
hence the moment generating
function exists if and only if $\Re z<1$.  
\end{remark}

We will derive our results using the (exact) formula \eqref{piglet}
for  moments. 
We will need a bound for the denominator there,
more precise than \refL{Lhw7}.

\begin{lemma}\label{LZ}
  Let $s=\gs+\ii\gt$ with $\gs>1$. Then
  \begin{align}
    \label{lz}
|\psi(s)-\psi(1)|\ge\Re\bigpar{\psi(s)-\psi(1)}\ge\psi(\gs)-\psi(1)>0.
  \end{align}
\end{lemma}
\begin{proof}
  The first inequality is trivial, and the last follows since \eqref{psi'}
  implies that $\psi'(z)>0$ for $z>0$. Finally, \eqref{psi'} also implies
  that if $\gs>0$ and $\gt>0$, then $\Im\psi'(\gs+\ii\gt)< 0$, and thus
$\Re\psi(\gs+\ii\gt)$ increases as $\gt$ grows from 0 to $\infty$;
the case $\gt\le0$ follows by symmetry.
\end{proof}

\begin{lemma}\label{Lmgf}
  Let $n\ge2$ and $-1<\gs<0$.
Then, for every complex $z$ with $\Re z <\psi(1-\gs)-\psi(1)$,
\begin{align}\label{qz1}
\E[e^{zD_n}] &=
1 - 
\frac{z}{2\pi\ii}\intgs\frac{\gG(s)\gG(n)/\gG(n+s)}{\psi(1-s)-\psi(1)-z}\dd s
,\end{align}
where the integral is absolutely convergent.
\end{lemma}
\begin{proof}
  Note first that $1-\gs>1$, so $b:=\psi(1-\gs)-\psi(1)>0$ by \refL{LZ},
which furthermore shows that 
$\Re\bigpar{\psi(1-s)-\psi(1)}\ge b$ for $s=\gs+\ii\gt$. 
Hence, the denominator in the integral in \eqref{qz1} is bounded away from 0
when $\Re z<b$, uniformly for $z$ in any compact subset.
It follows, using \eqref{hw6b},
that the integral in \eqref{qz1} converges absolutely and defines an
analytic function of $z$ in the half-plane $\Re z<b$.

Note also that $1-\gs<2$, and thus $b<\psi(2)-\psi(1)=1$;
hence the half-plane $\Re z<b$ is contained in the half-plane $\Re z<1$
where we know that $\E[e^{z D_n}]$ exists and is analytic.
By analytic continuation, it thus suffices to show \eqref{qz1} for small
$|z|$, and we will in the rest of the proof assume $|z|<b$.
In particular, $|z|<1$, and thus 
$\E[e^{|z| D_n}]<\infty$. Consequently we have by \eqref{piglet}, which
holds also for $k=1$ by \eqref{hw5},
\begin{align}
\E[e^{z D_n}]=1+\sumk \frac{z^k}{k!}\E[D_n^k]
=1-\frac{1}{2\pi\ii}
\sumk\intgs z^k\frac{\gG(s)\gG(n)/\gG(n+s)}{(\psi(1-s)-\psi(1))^k}
.\end{align}
We may interchange the order of summation and integration by Fubini's
theorem, which is justified by \eqref{hw6b} together with \refL{LZ} which gives
$|\psi(1-s)-\psi(1)|\ge b>|z|$; then \eqref{qz1} follows by summing the
geometric series.
\end{proof}

The next step is to shift the line of integration to the right.
To do this in general would require a study of the roots of $\psi(s)-\psi(1)=z$
in the complex plane. We consider for simplicity only 
the case of real $z$; extensions to complex $z$ are left to the reader.

\subsection{Real $z$} \label{SSmgfreal}

Consider the equation
\begin{align}\label{qx1}
  \psi(1+x)-\psi(1)=z
\end{align}
for a real $z$. By \refL{Lpsi}, and using the notation there, the roots are 
\begin{align}\label{qx2}
  \rho_i(z):=s_i\bigpar{z+\psi(1)}-1,
\qquad i=0,1,2,\dots,
\end{align}
with $\rho_0(z)\in(-1,\infty)$ and $\rho_i(z)\in(-(i+1),-i)$ for $i\ge1$.
We are mainly interested in $\rho(z):=\rho_0(z)$, the largest root;
thus $\rho(z)$ is the unique real number in $(-1,\infty)$ satisfying
\begin{align}\label{rho}
  \psi\bigpar{1+\rho(z)}-\psi(1)=z.
\end{align}
The function $\rho:(-\infty,\infty)\to(-1,\infty)$ 
is strictly increasing and continuous, in fact analytic
(since $\psi$ is on $(0,\infty)$), and \eqref{rho} shows that $\rho(0)=0$.
Furthermore, since \eqref{no4} yields $\psi(2)-\psi(1)=1$, 
\eqref{rho} also shows $\rho(1)=1$.
Hence, $\rho$ is a bijection $(-\infty,1)\to(-1,1)$.

Fix a real $z\in(-\infty,1)$. Then, as just said, $\rho(z)\in(-1,1)$.
The zeroes of the denominator in \eqref{qz1} are $s=-\rho(z)\in(-1,1)$
and $-\rho_1(z)<-\rho_2(z)<\dots$, with $-\rho_1(z)\in(1,2)$.
Write for convenience $\rho:=\rho(z)$, and
take $\gs\in(-1,\min(\rho,0))$. Then 
\begin{align}
\psi(1-\gs)-\psi(1)>\psi\bigpar{1-\rho}- \psi(1)=z, 
\end{align}
and thus \refL{Lmgf} applies. 

We now shift the line of integration in
\eqref{qz1} to some $\gs\in(1,-\rho_1(z))$; this is easily justified 
using \eqref{hw6b} and noting that the proof of \refL{Lhw7}\ref{Lhw7B}
also shows, more generally, that $|\psi(1-s)-\psi(1)-z|\ge c>0$
when $|\gt|\ge1$ for any $z\in\bbR$.
We pass two poles of the integrand, at $s=0$ and $s=-\rho$,
and we pick up $-2\pi\ii$ times the residues there.
We assume that $z\neq0$, since the case $z=0$ is trivial; then these two
poles are distinct and both are simple.
The residue at 0 of the integrand in \eqref{qz1} is simply $-1/z$, so the
contribution there
is $-1$, which cancels the constant 1 in \eqref{qz1}.
The main term comes from the residue at $-\rho$, which is
$-\gG(-\rho)\gG(n)/[\gG(n-\rho)\psi'(1+\rho)]$.
Hence we obtain, for any $1<\gs<-\rho_1(z)=1+|s_1(z-\gamma)|$,
\begin{align}\label{qz2}
\E[e^{zD_n}] &
=
-z
\frac{\gG(-\rho(z))\gG(n)}{\psi'(1+\rho(z))\gG(n-\rho(z))}
-\frac{z}{2\pi\ii}\intgs\frac{\gG(s)\gG(n)/\gG(n+s)}{\psi(1-s)-\psi(1)-z}\dd s
.\end{align}
Since $z<1$, we have $\rho_1(z)<\rho_1(1)$, and thus we may here always
choose $\gs$ as
\begin{align}\label{gs1}
\gsx:=  -\rho_1(1)=1-s_1(1+\psi(1))=1+|s_1(\psi(2))|
\doteq 1.457  
.\end{align}

We note also that as $z\to0$, we have $\rho(z)\to0$ and
\begin{align}\label{qx3}
-z\gG(-\rho(z))
=
\frac{z}{\rho(z)}\gG(1-\rho(z))
\to 
\frac{1}{\rho'(0)}
=\psi'(1)
,\end{align}
since $\psi'(1)\rho'(0)=1$ follows by 
differentiation of \eqref{rho}.
We thus interpret $-z\gG(-\rho(-z))=\psi'(1)$ for $z=0$ and note that then
\eqref{qz2} holds trivially for $z=0$ too.

This leads to the following result.

\begin{theorem}\label{Tmgf}
 For any real $z<1$, \eqref{qz2} holds for all $n\ge2$
with   $\gs=\gsx$ given by \eqref{gs1}.
Hence, 
\begin{align}\label{qz3}
\E[e^{zD_n}] &
=
\frac{-z\gG(-\rho(z))}{\psi'(1+\rho(z))}\frac{\gG(n)}{\gG(n-\rho(z))}
+O\bigpar{n^{-\gsx}}
\end{align}
and
\begin{align}\label{qz33}
\E[e^{zD_n}] &
=
\frac{-z\gG(-\rho(z))}{\psi'(1+\rho(z))}n^{\rho(z)}
\cdot\bigpar{1+O\bigpar{n^{-\min(1,\gsx+\rho(z))}}}
.\end{align}
Furthermore, \eqref{qz3} holds uniformly for $z<1-\gd$ for any $\gd>0$,
and \eqref{qz33} holds uniformly for $z$ in a compact subset of $(-\infty,1)$.
\end{theorem}

\begin{proof}
  We have already shown that \eqref{qz2} holds with $\gs=\gsx$.
Furthermore, by \eqref{gs1} and \eqref{qx1}--\eqref{qx2}, 
  \begin{align}
    \psi(1-\gsx)-\psi(1)=
\psi(1+\rho_1(1))-\psi(1)
=1
. \end{align}
Hence, if $s=\gsx+\ii\gt$, then
the denominator $\psi(1-s)-\psi(1)-z$ in the integral in \eqref{qz2} is $1-z>0$
when $\gt=0$, and as seen in the proof of \refL{Lpsi}, it is non-real for
all $\gt\neq0$; furthermore, 
$\Re\bigpar{\psi(1-s)-\psi(1)-z}\to+\infty$ as $\gt\to\pm\infty$ 
by \eqref{5.11.2}.
It follows, by continuity and compactness, that 
$\bigabs{\psi(1-s)-\psi(1)-z}$ is bounded below by some $c(\gd)>0$
for all such $s$ and $z\le 1-\gd$; 
moreover, $\Re\bigpar{\psi(1-s)-\psi(1)-z}\ge -C-z\ge |z|/2$ if $z\le-2C$.
It follows, using \eqref{hw6b} and  \cite[5.11.12]{NIST},
that the last term in \eqref{qz2} is 
$O\bigpar{\gG(n)/\gG(n+\gsx)}=O\bigpar{n^{-\gsx}}$ uniformly for
$z\le1-\gd$, which proves  \eqref{qz3}.

The mean value theorem yields, using \eqref{no2} and \eqref{5.11.2},
uniformly for $|\rho|\le1$,
\begin{align}
  \log\gG(n)-\log\gG(n-\rho)
=\psi\bigpar{n+O(1)}\rho=\rho\log n + O\bigpar{n\qw}.
\end{align}
Thus \eqref{qz33} follows from \eqref{qz3}, 
recalling that $|\rho(z)|<1$ for all $z<1$.
\end{proof}

Note that $\gsx>1>-\rho(z)$, so the exponent in the error term 
in \eqref{qz33} is always
negative, and in fact less that $1-\gsx=s_2(\psi(2))\doteq-0.457$.
Note also that the proof shows that
for any fixed $z$, the exponent may be improved.

In the following  two subsections we give some consequences of \refT{Tmgf}.

\subsection{A central limit theorem}
\label{sec:CLT}

As a corollary of \refT{Tmgf}, we obtain a new proof of the following CLT.
As mentioned in Section \ref{sec:summary}, this has been 
proved in \cite[Theorem 1.7]{beta1} by analysing a recursion for the moment
generating function; other proofs by different methods are given in 
\cite{beta2-arxiv}, \cite{iksanovCLT}, and \cite{kolesnik}.

\begin{Theorem}
\label{TNorm}
\begin{align}\label{tnorm1}
  \frac{D_n - \mu \log n}{\sqrt{\log n}} \dto  
\mathrm{Normal}(0, \sigma^2) \ \mbox{ as } n \to \infty 
\end{align}
where
\begin{align}\label{tnorm2}
 \mu := 1/\zeta(2) = 6/\pi^2 \doteq 0.6079 ; \quad 
\sigma^2 := 2 \zeta(3)/\zeta(2)^3 \doteq 0.5401
.\end{align}
\end{Theorem}

\begin{proof}
Let $z=z_n$ be any sequence of real numbers with $z_n\to0$.
Then \eqref{qz33} and \eqref{qx3} yield, as \ntoo,
\begin{align}\label{qy1}
  \E[e^{z_nD_n}]=n^{\rho(z_n)}\cdot\bigpar{1+o(1)}
=e^{\rho(z_n)\log n+o(1)}
.\end{align}
We have $\rho(0)=0$, and
the implicit function theorem shows that $\rho$ is analytic at 0 (and
everywhere). Hence,
\begin{align}
  \label{qy2}
\rho(z_n)
=z_n\rho'(0)+\tfrac12z_n^2\bigpar{\rho''(0)+o(1)}
.\end{align}
Fix $t\in\bbR$ and let $z_n:=t/\sqrt{\log n}$ (for $n\ge2$).
Then \eqref{qy1}--\eqref{qy2} yield
\begin{align}\label{qy3}
  \E e^{t(D_n-\rho'(0)\log n)/\sqrt{\log n}}&
=
 \E e^{z_n(D_n-\rho'(0)\log n)}
\\\notag&
= \exp\bigpar{\tfrac12z_n^2(\rho''(0)+o(1))\log n + o(1)}
\\\notag&
\to e^{\frac12t^2\rho''(0)}
\end{align}
as \ntoo.
Since this  holds for every real $t$, we thus have the CLT
 \eqref{tnorm1} with $\mu=\rho'(0)$ and $\gss=\rho''(0)$.

Finally, 
differentiating \eqref{rho} yields
$\rho'(z)=1/\psi'(1+\rho(z))$ and in particular
$\rho'(0)=1/\psi'(1)=1/\zeta(2)$ 
as already noted in \eqref{qx3},
and differentiating again  yields
\begin{align}\label{qy4}
  \rho''(0)=-\frac{\psi''(1)}{\psi'(1)^3}=\frac{2\zeta(3)}{\zeta(2)^3},
\end{align}
which completes the proof.
\end{proof}

\subsection{Large deviations}
\label{sec:LD}
As another corollary of \refT{Tmgf}, we obtain large deviation results
by the G\"artner--Ellis theorem. 
We follow \cite[Section 2.3]{DZ}, but note that $n$ there is replaced by the 
``speed'' $\log n$ (see \cite[Remark (a) p.~44]{DZ}).
Thus we define (for $n\ge2$)
\begin{align}\label{qd1}
  Z_n:=\frac{1}{\log n}D_n
\end{align}
and note that  \eqref{qz33} 
and \refR{Rmgf1} yield as \ntoo,
for any fixed $\gl\in\bbR$,
\begin{align}\label{qd2}
  \frac{1}{\log n}\log \E \bigsqpar{e^{(\log n)\gl Z_n}}
=
  \frac{1}{\log n}\log \E e^{\gl D_n}
\to \gL(\gl):=
  \begin{cases}
    \rho(\gl), & \gl<1,
\\
+\infty, & \gl\ge1.
  \end{cases}
\end{align}
The Fenchel--Legendre transform $\gLX$ of $\gL$ 
\cite[Definition 2.2.2]{DZ}
is defined by
\begin{align}\label{qd3}
  \gLX(x):=\sup_{\gl\in\bbR}\bigcpar{x\gl-\gL(\gl)}
=\sup_{-\infty<\gl<1}\bigcpar{x\gl-\rho(\gl)}.
\end{align}
Recall that $z\mapsto\rho(z)$ is a bijection of $(-\infty,1)$ onto $(-1,1)$,
which by \eqref{rho} is the inverse function of
\begin{align}\label{qd5g}
\rho\mapsto  g(\rho):=\psi(1+\rho)-\psi(1).
\end{align}
Hence, if we define
\begin{align}\label{qd5h}
h(x,\rho):=xg(\rho)-\rho,
\end{align}
then 
\eqref{qd3} yields
\begin{align}\label{qd4}
  \gLX(x) 
= \sup_{-1<\rho<1}\cpar{xg(\rho)-\rho}
= \sup_{-1<\rho<1}{h(x,\rho)}
= \sup_{-1<\rho\le1}{h(x,\rho)}
\end{align}
where the last equality holds by continuity.
We have
\begin{align}\label{hrho}
\frac{\partial}{\partial\rho}h(x,\rho)=x\psi'(1+\rho)-1
\end{align}
and for any fixed $x>0$, this is
by \eqref{psi'}
strictly decreasing from $+\infty$ to $x\psi'(2)-1=x(\zeta(2)-1)-1$
as $\rho$ grows from $-1$ to 1.
Hence, for $x>0$,
$h(x,\rho)$ is a concave function of $\rho$, and the supremum in 
\eqref{qd4} is attained at a unique $\rhox(x)\in(-1,1]$ given by
\begin{align}\label{rhox}
  \begin{cases}
\psi'(1+\rhox(x))=1/x,   & 0<x\le x_1:=(\zeta(2)-1)\qw,
\\
\rhox(x)=1,   & x\ge x_1.
  \end{cases}
\end{align}
For $x\le0$, $h(x,\rho)$ is a decreasing function of $\rho$, and thus
the supremum in \eqref{qd4} is obtained by letting $\rho\to-1$,
and thus $g(\rho)\to-\infty$. 
Thus we obtain, combining the cases,
\begin{align}\label{qd6}
  \gLX(x)=
  \begin{cases}
+\infty, & x<0,
\\
1, & x=0,
\\
xg(\rhox(x))-\rhox(x), & x>0.
  \end{cases}
\end{align}
In particular, for $x\ge x_1$, we have by \eqref{rhox} and \eqref{qd6}
\begin{align}\label{qd7}
\gLX(x)=xg(1)-1=x-1.  
\end{align}

It follows from \eqref{qd3} or \eqref{qd4} that $\gLX$ is convex
and that it is continuous on $[0,\infty)$.
Taking the derivative in \eqref{qd6} yields,
for $x>0$,
\begin{align}\label{qd8}
  \frac{\ddx}{\ddx x}\gLX(x)=
g(\rhox(x))+\bigpar{xg'(\rhox(x))-1}\rhox'(x)
=
g(\rhox(x)),
\end{align}
since \eqref{rhox} implies both $xg'(\rhox(x))-1=0$ for $0<x\le x_1$ 
and $\rhox'(x)=0$ for $x\ge x_1$.
(In particular, we see that $\gLX(x)$ is continuously differentiable also at
$x=x_1$.) 
The derivative is thus strictly increasing for $0<x\le x_1$, so $\gLX$ is
strictly convex on $[0,x_1]$, while $\gLX(x)$ is linear for $x\ge x_1$ as shown
already in \eqref{qd7}.
Hence, in the terminology of \cite[Definition 2.3.3]{DZ}, $y\in\bbR$ is an
\emph{exposed point} of $\gLX$ if  $0<y<x_1$.

We note that \eqref{qd8} yields
$ \frac{\ddx}{\ddx x}\gLX(x)=0$ when $\rhox(x)=0$ 
(since $g(0)=0$), which by \eqref{rhox} holds if and only if
$x=x_0:=\psi'(1)\qw=6/\pi^2$;
furthermore,
\eqref{qd6} shows that $\gLX(x_0)=0$.
Since $\gLX$ is convex, it thus attains its minimum at $x_0$, and the
minimum is 0 (as it is has to be).

The G\"artner--Ellis theorem \cite[Theorem 2.3.6]{DZ}
now implies Theorem \ref{Tdev1}, restated here. 
\begin{theorem}\label{Tdev}
As \ntoo, we have:
  \begin{align}\label{tdev1}
  \Pr(D_n<x\log n)& = n^{-\gLX(x)+o(1)}, \qquad\text{if}\quad 0<x\le x_0,
\\ \label{tdev2}
\Pr(D_n>x\log n)& =n^{-\gLX(x)+o(1)}, \qquad\text{if}\quad x_0\le x< x_1,
\\   \label{tdev3} 
  \Pr(D_n>x\log n)& \le n^{-\gLX(x)+o(1)}, \qquad\text{if}\quad x\ge x_1.
  \end{align}
\end{theorem}

\refT{Tdev} improves estimates for the upper tail in 
\cite[Theorem 1.4]{beta1}.
As a sanity check we note that \eqref{tdev1} and \eqref{tdev2} imply 
that $D_n$ is  concentrated at $x_0\log n=\frac{6}{\pi^2}\log n$,
and in particular $D_n/\log n\pto \frac{6}{\pi^2}$, which is proved in
\cite{beta1}, and also follows from \refT{TD} and \eqref{kb9}, or from
\refT{TNorm}.

\begin{remark}
  The results in this section are based on \eqref{qz1} which is obtained by
  summing the corresponding results for the moments $\E[D_n^k]$.
For real $z$, an alternative is to define the (signed) measure
  \begin{align}\label{axz0}
    \Xi_z:= z\intoo e^{zt}\cL(P_{t,1}) \dd t
.  \end{align}
(This is an infinite positive measure for $z>0$ 
and a finite negative measure for $z<0$.)
Then
\begin{align}
  \E e^{zD_n}= 1 + \intoo ze^{zt}\Pr(D_n>t)\dd t
\end{align}
and thus \eqref{cx3}  yields
\begin{align}\label{axz1}
\E e^{zD_n}&
=1+\intoi\bigsqpar{1-(1-x)^{n-1}}\dd \Xi_z(x).
\end{align}
Furthermore, we obtain from \eqref{axz0} and \eqref{b5} the Mellin transform,
for $\Re s=\gs>1$ and $\psi(\gs)>\psi(1)+z$,
\begin{align}\label{axz4}
\intoi x^{s-1}\dd\Xi_z(x)&
= 
z\intoo e^{t(z-\psi(s)+\psi(1))}\dd t 
=\frac{z}{\psi(s)-\psi(1)-z}
.\end{align}
For $z<1$ we may then by the methods above obtain \eqref{qz1} and
\eqref{qz2}. 

For our purposes, we prefer the method above, but we mention the alternative
since it might have other uses.  
\end{remark}

\section{Final remarks}
As mentioned in the introduction, this article is part of a broad project investigating different aspects of the random tree model.
The document \cite{beta2-arxiv} is intended to maintain a current overview of the project, summarizing known results, open problems and heuristics, 
and indicating the range of proof techniques. 
Let us mention just two aspects related to this article.

{\bf 1.}  The definition of $D_n$ involves two levels of randomness: 
the realization of the tree and the distribution of leaf heights within that realization. 
A start at quantifying this feature is a ``lack of correlation" result in \cite[Theorem 1.6]{beta1}.

{\bf 2.} Heuristics for the height of $\CTCS(n)$, that is the maximum leaf height in a realization, are discussed in \cite{beta2-arxiv}.
This problem seems surprisingly subtle: the naive guess based on the Normal approximation (Theorem \ref{TNorm1}) is definitely incorrect.

\section*{Acknowledgments}
Thanks to Boris Pittel for extensive interactions regarding this project.
 Thanks to B\'{e}n\'{e}dicte Haas for his careful explanation of how our setting fits into the general theory of exchangeable random partitions.


\appendix

\section{Proof of \refL{LP}}\label{AParseval}

We repeat the statement of \refL{LP}.

\newtheorem*{LemmaLP}{Lemma \ref{LP}}
\begin{LemmaLP}
Suppose that $f$ is a locally integrable function and $\mu$ a measure
on $\bbR_+$,
and that $\gs\in\bbR$ is such that $\intoo x^{\gs-1}|f(x)|<\infty$ 
and $\intoo x^{-\gs}\dd\mu<\infty$, i.e., the
Mellin transforms $\Mf(s)$ and $\Mmu(1-s)$ are defined  when 
$\Re s=\gs$. 
Suppose also  that the integral $\intgs\Mf(s)\M\mu(1-s)\dd s$ converges
at least conditionally.
Suppose further that $x^\gs f(x)$ is bounded and that $f$ is 
$\mu$-a.e.\ continuous.
Then 
\begin{align}
 \intoo f(x)\dd\mu(x) = \frac{1}{2\pi\ii}\intgs\Mf(s)\M\mu(1-s)\dd s.
\tag{\ref{lp}}
\end{align}
\end{LemmaLP}

Although we have not found \refL{LP} in the literature,
it is closely related to standard versions of Parseval's formula for Mellin
transforms. These are usually stated for two functions, while our lemma is
stated for a function and a measure; however, for absolutely continuous
$\mu$, our lemma reduces to a result for two functions.
In particular, in the case that $\mu$ is absolutely continuous, \refL{LP} 
is a special case of the more general \cite[Theorem 43]{Titchmarsh}.
(Note that we have shown in \refL{LM} that our measure $\gU$ is absolutely
continuous, and as discussed in \refS{SSgUk}, this extends to the measures
$\gU_k$ considered there; however, 
one advantage of our use of \refL{LP} is that we do not have to verify this
by a special argument.)
For completeness, we give a detailed proof below.

Note also that by a standard change of variables, as in the proof below,
any version of Parseval's formula for Mellin transforms is equivalent to 
a corresponding version of Parseval's formula (also called Plancherel's
formula)  for the Fourier transform on the real line;
see e.g.\ \cite{Titchmarsh}.
(From the point of view of abstract Harmonic analysis, the Mellin transform
is just the Fourier transform on the multiplicative group $(\bbR_+,\cdot)$
with Haar measure $\ddx x/x$.)

\begin{proof}[Proof of \refL{LP}]
We use a well-known transformation to the Fourier transform on $\bbR$,
which we do in two steps.
First, let $f_\gs$ and $\mu_\gs$ be the function and measure on $\bbR_+$
defined by
  \begin{align}
f_\gs(x):=x^{\gs}f(x),\qquad
    \dd\mu_{-\gs}(x) := x^{-\gs}\dd\mu(x)
  \end{align}
(i.e., the Radon--Nikodym derivative $\ddx\mu_{-\gs}/\dd\mu=x^{-\gs}$).
Note that 
then $\Mf_\gs(s)=\Mf(s+\gs)$ and $\Mmu_\gs(s)=\Mmu(s-\gs)$,
and that
$f_\gs$ and $\mu_\gs$ satisfy the assumptions of \refL{LP} with
$\gs=0$, and 
so we have reduced the lemma to this case.
(In particular, $\mu_\gs$  is a finite measure.)
Secondly, we change variables to $y\in\bbR$ by the mapping $y=\log x$;
we let $F(y):=f_\gs(e^y)$ and let $\nu$ be the (finite) measure on $\bbR$
corresponding to $\mu_\gs$. Then the assumptions on $f$ and $\mu$ imply
that $F\in L^1(\bbR)$, $F$ is bounded, $F$ is $\nu$-a.e.\ continuous,
and that $\nu$ is a finite measure on $\bbR$. Furthermore,
for any $\tau\in\bbR$,
\begin{align}
  \Mf(\gs+\ii\tau)
=\intoo f(x) x^{\gs+\ii\tau-1}\dd x
=\intoo f_\gs(x) x^{\ii\tau-1}\dd x
=\intoooo F(y) e^{\ii\tau y}\dd y =: \hF(\tau),
\end{align}
the Fourier transform of $F$,
and
\begin{align}
  \Mmu(1-(\gs+\ii\tau))
=\intoo x^{-\gs-\ii\tau}\dd \mu(x)
=\intoo  x^{-\ii\tau}\dd \mu_\gs(x)
=\intoooo e^{-\ii\tau y}\dd \nu(y) 
=: \hnu(-\tau)
.\end{align}

Hence, \refL{LP} follows from (and is equivalent to)
the following lemma for the Fourier transform on $\bbR$.
\end{proof}

\begin{lemma}\label{LT1}
Suppose that $f$ is an integrable function on $\bbR$
and that $\nu$ is a finite 
measure on $\bbR$
such that the integral $\intoooo\hf(t)\hnu(-t)\dd t$ converges
at least conditionally.
Suppose further that $f$ is bounded and $\nu$-a.e.\ continuous.
Then 
\begin{align}\label{parseval2}
 \intoooo f(x)\dd\nu(x) = \frac{1}{2\pi}\intoooo\hf(t)\hnu(-t)\dd t.
\end{align}
\end{lemma}

\begin{proof}
 We begin by taking
the F{\'e}jer kernel on $\bbR$, for $s>0$,
\begin{align}\label{ks1}
  k_s(x) :=
\frac{1-\cos(sx)}{\pi sx^2}=
\frac{\sin^2(sx/2)}{2\pi s(x/2)^2},
\end{align}
which is integrable and 
has the Fourier transform
\begin{align}\label{ks2}
  \hks(t)=(1-|t|/s)_+,
\end{align}
see e.g.\ \cite[VI.(1.8), p.~124]{Katznelson}.
Thus the convolution $k_s*f$ has Fourier transform
\begin{align}\label{ks3}
  \widehat{k_s*f}(t)=\hf(t)\hks(t)=(1-|t|/s)_+\hf(t).
\end{align}
The function $\hf(t)$ is bounded, since $f\in L^1$, and thus
\eqref{ks3} shows that $  \widehat{k_s*f}$ is integrable.
Consequently, the Fourier inversion formula applies and yields, 
for every $s>0$ and a.e.\ $x\in\bbR$,
\begin{align}\label{ks4}
  k_s*f(x) = \frac{1}{2\pi}\intoooo e^{-\ii tx}\hf(t)(1-|t|/s)_+\dd t.
\end{align}

We obtain from \eqref{ks4}
by Fubini's theorem (twice), since the double integrals are
absolutely convergent by the assumptions that $f\in L^1$ and $\nu$ is finite,
and thus $\hf$ and $\hnu$ are bounded,
\begin{align}\label{ks5b}
  \intoooo k_s*f(x)\dd \nu(x)&
=\intoooo \frac{1}{2\pi}\intoooo e^{-\ii tx}\hf(t)(1-|t|/s)_+\dd t \dd\nu(x)
\\\notag&
= \frac{1}{2\pi}\int_{-s}^s \hnu(-t)\hf(t)(1-|t|/s)\dd t
\\\notag&
=\frac{1}{2\pi}\intoi \int_{-us}^{us}\hnu(-t)\hf(t)\dd t \dd u
.\end{align}
The assumption that $\intoooo \hf(t)\hnu(-t)\dd t$ converges conditionally
means that as $s\to\infty$, the inner integral
on the last line of \eqref{ks5b} tends to
$J:=\intoooo\hf(t)\hnu(-t)\dd t$;
it follows also that 
this inner integral is bounded uniformly in $u$ and $s$.
Hence, \eqref{ks5b} and dominated convergence show that as \stoo,
\begin{align}\label{ks6b}
  \intoooo k_s*f(x)\dd\nu(x)&
\to 
\frac{1}{2\pi}\intoi J \dd u=\frac{1}{2\pi} J
.\end{align}
Moreover, since $f$ is bounded, it is easily seen (and well-known)
that, as \stoo, 
$k_s*f(x)\to f(x)$ for every $x$ such that $f$ is continuous at $x$;
by assumption this holds for $\nu$-a.e.\ $x$.
Since furthermore $k_s*f$ is bounded (by $\sup_x|f(x)|$),
it follows by dominated convergence that
\begin{align}\label{ks7b}
    \intoooo k_s*f(x)\dd \nu(x)
\to \intoooo f(x)\dd\nu(x)
\qquad \text{as } \stoo
.\end{align}
Finally, \eqref{parseval2} follows from \eqref{ks6b} and \eqref{ks7b}.
\end{proof}

\begin{remark}
  \refL{LT1} extends, with the same proof, to complex measures $\nu$. 
\end{remark}

\begin{remark}
As far as we know, it is unknown 
(even in the absolutely continuous case)
whether 
\eqref{parseval2} holds for any integrable $f$ and finite measure $\nu$
such that both integrals in \eqref{parseval2} converge.
\end{remark}

\end{document}